\documentclass[reqno]{amsart}

\usepackage[english, activeacute]{babel}
\usepackage{amsmath,amsthm,amsxtra}
\usepackage{fix-cm}
\usepackage{graphics}
\usepackage{fancyhdr}
\usepackage{float}
\usepackage{anysize}
\usepackage{amssymb,mathrsfs}
\usepackage{enumerate}
\usepackage{epsfig}
\usepackage{latexsym}
\usepackage{amsfonts}
\usepackage{hyperref}
\usepackage{dsfont}
\usepackage{color}
\usepackage{tikz,pgfplots,pgf}
\tikzset{node distance=2cm, auto}
\usetikzlibrary{calc}
\usetikzlibrary{trees}
\usepackage{epstopdf}
\usepackage{xfrac}
\usepackage{tikz}
\usetikzlibrary{automata,positioning,matrix}
\usepackage[normalem]{ulem}

\makeatletter
\@namedef{subjclassname@2020}{%
  \textup{2020} Mathematics Subject Classification}
\makeatother

\newtheorem{theorem}{Theorem}[section]
\newtheorem*{theorem*}{Theorem}
\newtheorem{corollary}[theorem]{Corollary}

\newtheorem{lemma}[theorem]{Lemma}
\newtheorem{proposition}[theorem]{Proposition}

\theoremstyle{definition}
\newtheorem{definition}[theorem]{Definition}
\newtheorem{remark}[theorem]{Remark}

\newtheorem{example}[theorem]{Example}
\newtheorem{examples}[theorem]{Examples}

\newcommand{\N}{\mathbb{N}}

\def\R{\mathbb{R}}
\def\F{\mathcal{F}}

\def \Lip{\mathrm{Lip}}
\def\SLip{\mathrm{SLip}}

\begin{document}




\title[Asymmetric free spaces]{Asymmetric free spaces\\ and canonical asymmetrizations}

\author[A. Daniilidis]{Aris Daniilidis}

\author[J.M. Sepulcre]{Juan Mat\'{i}as Sepulcre}

\author[F. Venegas M.]{Francisco Venegas M.}

\date{}

\begin{abstract}
A construction analogous to that of Godefroy-Kalton for metric spaces allows to embed isometrically,  in a canonical way, every quasi-metric space $(X,d)$ to an asymmetric normed space $\mathcal{F}_a(X,d)$ (its quasi-metric free space, also called asymmetric free space or semi-Lipschitz free space). The quasi-metric free space satisfies a universal property (linearization of semi-Lipschitz functions). The (conic) dual of $\mathcal{F}_a(X,d)$ coincides with the nonlinear asymmetric dual of $(X,d)$, that is, the space $\SLip_0(X,d)$ of semi-Lipschitz functions on $(X,d)$, vanishing at a base point.
In particular, for the case of a metric space $(X,D)$, the above construction yields its usual free space.
On the other hand, every metric space $(X,D)$ inherits naturally a canonical asymmetrization coming from its free space $\mathcal{F}(X)$. This gives rise to a quasi-metric space $(X,D_+)$ and an asymmetric free space $\mathcal{F}_a(X,D_+)$. The symmetrization of the latter is isomorphic to the original free space $\mathcal{F}(X)$. The results of this work are illustrated with explicit examples.
\end{abstract}



\maketitle

\tableofcontents

\section{Introduction}
\noindent Arens and Eells \cite{AE1956} showed that every metric space
$(X,D)$ can be isometrically embedded as a closed subset of a normed linear space. The closed linear span of the image of $X$ under this embedding is called the {\em Arens-Eells space}, see  \cite[Section~2.2]{Weaver}. The idea of considering isometric embeddings of metric spaces to linear spaces goes back to Kuratowski, Wojdy{s\l}awski and Klee~\cite{Klee_1951}. It also implicitly appears in classical works of Kantorovich~\cite{Kantorovich42, Kantorovich57}, where a new distance in the space of finite measures on $X$ was defined (known as the Kantorovich-Rubinstein distance, \cite{Hanin92} {\it e.g.}) in a way that entails an isometric embedding of $(X,D)$ into the dual space $C(X)^{*}$. In~\cite{AE1956}, the authors also obtained an analogous embedding of a uniform space in a locally convex linear space. The terms \emph{free Banach space} and respectively, \emph{free locally convex space} have then been conceived (\cite{raikov64, pestov93}) to refer to the resulting spaces. \smallskip \\
The terminology \emph{Lipschitz-free space} (or simply, \emph{free space}) over a metric space $(X,D)$ has been introduced and highly popularized with the seminal work of Godefroy and Kalton~\cite{GK_2003}, where the authors employed this term to describe a very similar construction to the one of the Arens-Eells space (see Remark~\ref{rem-AE} for a comparison), but with emphasis on the linearization of both the metric space and its natural morphisms (Lipschitz functions between metric spaces). Free spaces, under this new terminology, have rapidly gained the interest of many researchers in Functional Analysis (\cite{AK2009,APP,Borel2012,CDW2016,toni,Godard,Hajek2014}~{\em e.g.}) and the topic became, arguably, one of its most active trends nowadays.  \smallskip \\
Let us outline below the construction. Given a metric space $(X,D)$ with a distinguished point $x_0$ (called \emph{base point}), the free space $\F(X)$ is constructed as follows: we first consider as \emph{pivot} space (\emph{non-linear dual} of $X$) the Banach space $\mathrm{Lip}_0(X)$ of real-valued Lispchitz functions vanishing at the base point, endowed with the norm $$\|f\|_{\mathrm{Lip}}=\displaystyle\sup_{\substack{x,y\in X\\x\neq y}}\frac{|f(x)-f(y)|}{D(x,y)}.$$
Then each $x\in X$ is identified to a Dirac measure $\delta_x$ acting linearly on $\mathrm{Lip}_0(X)$ as evaluation. Then the mapping $$\widehat{\delta}: X \mapsto {\Lip_0(X)}^{*}$$ that maps $x$ to $\delta_x$ is an isometric embedding. The Lipschitz-free space $\mathcal{F}(X)$ over $X$ is defined as the closed linear span of $\widehat{\delta}(X)$ in $\mathrm{Lip}_0(X)^*$. Furthermore, the free space is a predual for $\mathrm{Lip}_0(X)$, meaning that $\mathcal{F}(X)^*$ is isometrically isomorphic to $\mathrm{Lip}_0(X)$ (therefore, the space $\mathrm{Lip}_0(X)$ is at the same time the (linear) dual of $\mathcal{F}(X)$ and the nonlinear dual of $X$).
For a survey on the properties and development of Lipschitz-free spaces, we refer the reader to \cite{G_2015}. We also refer to Bachir~\cite{Bachir2004} for prior constructions based on evaluations over some algebra of functions acting on $X$. \smallskip \\
In this work, using the aforementioned embedding, we show that metric-spaces can be asymmetrized in a canonical way, giving rise to quasi-metric spaces, that is, spaces equipped with an asymmetric distance (see forthcoming Definition~\ref{deftop}).
Semi-Lipschitz functions (Definition~\ref{slip}) are the natural morphisms for such spaces. Starting from a quasi-metric space $(X,d)$ with a base point $x_0\in X$, the normed cone structure (Definition~\ref{defnorm}) of the set $\SLip_0(X)$ of real-valued semi-Lipschitz functions on $X$, vanishing at $x_0$ is used as an {\it asymmetric pivot} space to obtain a semi-Lipschitz free construction, which is analogous to the Kalton-Godefroy symmetric construction (this latter uses as pivot the Lipschitz functions). This leads to an adequate notion of {\em semi-Lipschitz free space} (or {\em quasi-metric} free space) $\mathcal{F}_a(X,d)$ for $(X,d)$, where the set $\SLip_0(X)$ is both the nonlinear (conic) dual of $X$ and the (linear, conic) dual of $\mathcal{F}_a(X,d)$. We emphasize the fact that $\SLip_0(X)$ is not a linear space in general, therefore we need to enhance the duality of \emph{normed cones}. This being said, the semi-Lipschitz free construction remains compatible with the classical one in the symmetric case. Moreover, it is also compatible with the aforementioned canonical asymmetrization, in the sense that the semi-Lipschitz free space of the canonical asymmetrization of a metric space and the asymmetrization of its free space are often identical (Proposition~\ref{SS*}) and in any case they have isomorphic symmetrizations (Theorem~\ref{Nino}).
\smallskip

Quasi-metric spaces and asymmetric norms have recently attracted a lot of interest in modern mathematics: they arise naturally when considering non-reversible Finsler manifolds \cite{CJ,DJV,Gaubert3} (see also \cite{BCS, DH}), and meet applications in physics \cite{JLP}, as well as in game theory \cite{Gaubert1,Gaubert2}. The properties of spaces with asymmetric norms have been studied by several authors (see \cite{cobzas2012functional}, \cite{romaguera2005properties} and references therein), emphasizing similarities and differences with respect to the theory of (symmetric) normed spaces. Besides its intrinsic interest, and the aforementioned applications, this theory was also stimulated by the study of oriented graphs and by applications in Computer Science, mainly to the complexity of algorithms. In this work we endeavor a new insight
in the current {\it State-of-the-art}, by showing that morphisms of quasi-metric spaces can be linearized in a similar manner as in the symmetric case through an asymmetric free space, and that this asymmetric free theory behaves equally well and it is fully compatible to the symmetric theory in a canonical manner. Indeed, there is a canonical way to move from a symmetric to an asymmetric space and vice-versa, which in addition, is compatible with the embeddings to their free spaces.
\smallskip

The manuscript is organized as follows: in Section~\ref{prelim} we recall basic notions, definitions and we fix our notation. We also give some auxiliary results required for the development of the theory in the asymmetric case, together with results about linear functionals, dual conic-norms and continuity on normed cones. We also give the definition of a canonical asymmetrization of a metric space. The main result will be established in Section~\ref{main}, with the definition of the semi-Lipschitz free space $\F_a(X)$ of a quasi-metric space $X$ (Definition~\ref{free}) and its characteristic feature that its dual is exactly the space $\SLip_0(X)$ ({\em c.f.} Theorem~\ref{predual}). The semi-Lipschitz free space $\F_a(X)$ is a bi-complete asymmetric normed space (it is naturally endowed with an asymmetric norm). For this reason, we shall also refer to it as the asymmetric free space of $X$. In Section~\ref{linearization}, through a simple diagram chasing argument, we shall show that semi-Lipschitz free spaces enjoy a canonical (and useful) linearization property: every semi-Lipschitz map between pointed quasi-metric spaces extends to a linear map between the corresponding semi-Lipschitz free spaces (see Corollary~\ref{linecor}). In Section~\ref{examples} we give concrete examples of asymmetric free spaces in order to help the reader to get an insight for this new theory. Finally, Section~\ref{sec:6} contains open questions and outlines possible research lines for further investigation on this topic.
\medskip

\section{Notation, Preliminaries}\label{prelim}

\noindent Throughout this article we denote by $\mathbb{R}_+$ the set of non-negative real numbers and we use the convention $\inf \emptyset=+\infty$. Given a vector space $E$, we denote by $\|\cdot\|:E\to \R_+$ a norm on $E$ and by $\|\cdot|:E\to\R_+$ an {\em asymmetric norm} on $E$, that is, a function satisfying:
\begin{itemize}
\item[(i)] $\forall x,y\in E$:\, $\|x+y|\,\leq\,\|x|\,+\,\|y|$;  \smallskip
\item[(ii)] $\forall x\in E$:\quad $x=0 \iff \|x|=0$;  \smallskip
\item[(iii)] $\forall x\in E,\,\forall r > 0$:\, $\|r\,x|\,=r\,\|x|$.
\end{itemize}
If we replace the second condition by
$$
{\rm (ii)' }\quad x=0\quad\iff\quad \begin{cases}\
\|x|=0\\[0.15cm]
\ \|\!-\!\!x|=0
\end{cases}
$$
then we say that $\|\cdot|:E\to\R_+$ is an {\em asymmetric hemi-norm} on $E$. The terminology of \emph{asymmetric normed space} refers to pairs $(E,\|\cdot|)$ having either asymmetric norms or asymmetric-hemi norms on $E$. Notice that an asymmetric (hemi-)norm differs from a norm in the fact that the equality $\|\!-\!x|=\|x|$ is not necessarily true.  \smallskip

\noindent We may also consider, keeping the same notation, {\it extended asymmetric norms}, allowing $\|\cdot|$ to take the value $+\infty$. Finally, we denote by $u$ the asymmetric hemi-norm on $\R$ defined by
\begin{equation}\label{u}
u(x)=\max\{x,0\}, \quad\text{for every } x\in\R.
\end{equation}

\begin{remark}[Asymmetrizations in $\mathcal{F}(X)$]
\label{naturalasymmetrizednorm} There is a natural way to asymmetrize the
norm $\Vert \cdot \Vert _{\mathcal{F}}$ of the free space $\mathcal{F}(X)$ of
a given metric space $(X,D)$, based on the dual space $L:=\mathrm{Lip}_{0}(X).$ Let us denote by
$\langle \cdot ,\cdot \rangle $ the duality map
of the duality pair $(L,\mathcal{F}(X)) $. Then the norm
$||\cdot ||_{\mathcal{F}}$ of $\mathcal{F}(X)$ can be represented as follows:
\begin{equation}
\Vert Q\Vert _{\mathcal{F}}:=\sup_{\substack{ \phi \in L  \\ \Vert \phi
\Vert _{L}\leq 1}}\langle \phi ,Q\rangle ,\ \mbox{ for every }Q\in \mathcal{F}(X).  \label{kcwooeh}
\end{equation}
Let us recall that a (convex) cone in a linear space is a (convex) subset $P$ such that $\lambda x \in P$ for every $x\in P$ and $\lambda \in \mathbb{R}_+$. In this work we shall use the term {\it cone} to refer to a convex cone. Consider any generating closed cone $P$ of $L$ (i.e., $L=\mathrm{span}\left(
P\right) =P-P$) that satisfies:
\begin{equation}
\forall \phi \in L,\;\exists \phi _{1},\phi _{2}\in P:\;\left\{
\begin{array}{c}
\phi =\phi _{1}-\phi _{2}\medskip \\
\max \,\left\{ ||\phi _{1}||_{L},||\phi _{2}||_{L}\right\} \,\leq \,||\phi
||_L\,\leq \,||\phi _{1}||_{L}+||\phi _{2}||_{L}
\end{array}.
\right.  \label{eq:P}
\end{equation}
We set:
\begin{equation}
\Vert Q|_{\mathcal{F}_{P}}:=\sup_{\substack{ \phi \in P  \\ \Vert \phi \Vert
_{L}\leq 1}}\langle \phi ,Q\rangle ,\ \mbox{ for every }Q\in \mathcal{F}(X).
\label{eq_F_P-norm}
\end{equation}

Notice that for any $Q\in \mathcal{F}(X)$ we have $\max \left\{ \Vert Q|_{\mathcal{F}_{P}},\Vert\!\!-\!\!Q|_{\mathcal{F}_{P}}\right\} \leq \Vert Q\Vert _{\mathcal{F}}$. Since the supremum in \eqref{kcwooeh} is attained at some $\phi \in L$ with $||\phi ||_{L}=1$ (by Hahn-Banach theorem), using the
decomposition (\ref{eq:P}) we deduce:
\begin{equation}
\Vert Q\Vert _{\mathcal{F}}=\langle \phi ,Q\rangle =\langle \phi
_{1},Q\rangle +\langle \phi _{2},-Q\rangle \leq \Vert Q|_{\mathcal{F}_{P}}+\Vert\!\!-\!\!Q|_{\mathcal{F}_{P}}.  \label{eq:equiv}
\end{equation}
This shows that $(ii)^{\prime }$ holds and \eqref{eq_F_P-norm} defines an
asymmetric (hemi-)norm $\Vert \cdot |_{\mathcal{F}_{P}}$ on the vector space
$\mathcal{F}(X)$.\smallskip

We shall refer to the asymmetric norm $\Vert \cdot|_{\mathcal{F}_{P}}$ defined in~\eqref{eq_F_P-norm} as the $P$-\emph{asymmetrization} of the free space $\mathcal{F}(X)$, for which we implicitly assume that~\eqref{eq:P} holds. \smallskip
We shall mainly deal with the case where $P$ is the cone of positive Lipschitz functions,
that is:
\begin{equation*}
P=L_{+}:=\{\phi \in L:\phi \geq 0\}.
\end{equation*}
In this case, we denote the arising asymmetric norm by $\Vert \cdot |_{\mathcal{F}_{+}}$. Notice that if $\phi(=\phi^+-\phi^-)\in L$ then both its positive part $\phi^+$ and its negative part $\phi^-$ are also in $L$ and they satisfy $|\phi^+(x)-\phi^+(y)|\leq |\phi(x)-\phi(y)|$ and $|\phi^-(x)-\phi^-(y)|\leq |\phi(x)-\phi(y)|$, for all $x,y\in X$, which leads to~\eqref{eq:P}.\smallskip

\noindent More generally, a $P$-asymmetrization of $\mathcal{F}(X)$ is called canonical,
if $P$ is of the form
\begin{equation*}
P:=\left\{ \phi \in L:\,T\phi \geq 0\right\} ,
\end{equation*}
where $T$ is a {\em linear isometry} that identifies $L$ with some
Banach lattice in a canonical way.
\end{remark}

\subsection{Quasi-metric spaces}
Let us introduce the notion of a quasi-metric space, which will be the main focus of this work.
\begin{definition}[Quasi-metric space]\label{deftop}
A \emph{quasi-metric space} is a pair $(X,d)$, where $X\neq\emptyset$ and $${d:X\times X\to [0,\infty)}$$
is a function, called {\em quasi-metric} (or {\em quasi-distance}), satisfying: \smallskip
\begin{enumerate}
\renewcommand\labelenumi{(\roman{enumi})}
\leftskip .35pc
     \item $\forall x,y,z\in X$: $d(x,y)\leq d(x,z)+d(z,y)$ (\emph{triangular inequality});\smallskip
     \item $\forall x,y\in X$: $x=y$ $\iff$ $d(x,y)=0$.
\end{enumerate}
\end{definition}
\noindent Note that a quasi-metric does not possess the symmetric property $d(x,y)=d(y,x)$ of a distance. \\
If we replace the last condition by
$$
{\rm (ii)' }\qquad x=y\quad\iff\quad \begin{cases}
\ d(x,y)=0\\[0.15cm]
\ d(y,x)=0
\end{cases}
$$
then we say that $d$ is a {\em quasi-hemi-metric}.
In this work we shall also consider \emph{extended quasi-metrics} $\tilde{d}:X\times X\to [0,\infty]$, that is, quasi-metrics that satisfy the same two conditions above, but are also allowed to take the value $+\infty$. If $X$ is a vector space equipped with an (extended) asymmetric (hemi-)norm $\|\cdot|$, then the function
\begin{equation}\label{qnqd}
d(x,y):=\|y-x|,\quad\text{for all } x,y\in X
\end{equation}
is an (extended) quasi-(hemi-)metric on $X$ that satisfies:
\begin{equation}\label{invariant}
d(x+z,y+z)=d(x,y)\quad\text{and}\quad d(rx,ry)=rd(x,y),
\end{equation}
for all $x,y,z\in X$ and $r\in \mathbb{R}_{+}$. Furthermore, for every $x,y\in X$ the \emph{reverse quasi-metric} $\bar{d}$ is defined by $$\bar{d}(x,y)=d(y,x).$$
Throughout this paper, we shall treat both variants of quasi-metric spaces. The terminology of quasi-metric space will thus refer to a pair $(X,d)$ where $d$ is either a quasi-distance or a quasi-hemi-distance.

\begin{remark}[Terminology alert I]
The reader should be alerted that terminology may slightly vary according to the authors. Some authors allow the quasi-hemi-metric and the asymmetric hemi-norm to also take negative values. They also use the terms hemi-metric and hemi-norm to refer to what we call quasi-hemi-metric and asymmetric hemi-norm, respectively (see, for instance, \cite{Gaubert2}).
In our work, the adjective \textit{quasi} refers to the asymmetry of the metric, and the adjective \textit{hemi} to the fact that distinct elements $x, y$ in $X$ may have a quasi-distance $d(x,y)$ equal to $0$.
\end{remark}

Two quasi-metric spaces can be completely identified via isometries. (The reader should be alerted that the slightly weaker notion of {\it almost isometry} also exists, and is more appropriate in relation with Banach-Stone type theorems, see \cite{CJ, DJV}.)

\begin{definition}[Isometry]\label{defiso}
A bijective mapping $\Phi$ between extended quasi-metric spaces $(X,d)$ and $(Y,\rho)$ is called an \emph{isometry} if for every $x_1,x_2\in X$, it holds
\begin{equation*}\label{nno}
\rho(\Phi(x_1),\Phi(x_2))=d(x_1,x_2).
\end{equation*}
\end{definition}
\smallskip

\begin{definition}[Canonical asymmetrization of a metric space]
\label{canoasym} Let $(X,D)$ be a metric space with a base point $x_0\in X$. Every $P$-asymmetrization of
the free space $\mathcal{F}(X)$ ({\it c.f.} Remark~\ref{naturalasymmetrizednorm})
induces, via the isometric injection of $X$ into $\mathcal{F}(X),$ an
asymmetrization of the distance $D$, given by:
\begin{equation*}
D_{P}(x,y)=\Vert {\delta}_y-{\delta}_x|_{\mathcal{F}_{P}}=\sup_{\substack{ \phi
\in P  \\ \Vert \phi \Vert _{L}\leq 1}}(\phi (y)-\phi (x)),\ \quad \text{for
all }x,y\in X.  \label{eq_D_P}
\end{equation*}
The quasi-(hemi-)distance $D_P$ is called the $P$-asymmetrization of $(X,D).$ If $\Vert \cdot |_{\mathcal{F}_{P}}$ is a canonical asymmetrization
of $\mathcal{F}(X),$ then $D_{P}$ will be called a \emph{canonical
asymmetrization} of $D$. In case $P=L_{+}$, the canonical asymmetrization
will be denoted by $D_{+}.$ The diagram below illustrates the situation.
\end{definition}\bigskip

\begin{center}
\begin{tabular}{cccc}
\cline{2-3}\cline{2-3}
 & \multicolumn{1}{|c}{} & &  \multicolumn{1}{|c}{} \\
 &\multicolumn{1}{|c}{} & \multicolumn{1}{c}{$L=\Lip_0(X,D)$} & \multicolumn{1}{|c}{} \\
\multicolumn{1}{c|}{} & &  & \multicolumn{1}{|c}{}  \\
 & \multicolumn{1}{|c}{} & \multicolumn{1}{c}{$\|\phi\|_L=\underset{x\neq y}{\mathrm{sup}}\frac{\phi(x)-\phi(y) }{ D(x,y)}$} & \multicolumn{1}{|c}{}  \\
& \multicolumn{1}{|c}{} & & \multicolumn{1}{|c}{}  \\
\cline{2-3}\cline{2-3} &  &  \\
 \hspace{0.2cm} $\shortstack[cc]{{\ \ Nonlinear}\\{\ \ dual}} $ \Huge$\nearrow $  &  &  &  {\Huge$\nwarrow $} $\shortstack[cc]{{Linear\ \ \ \ \ \ \ }\\{dual\ \ \ \ \ \ \ }} $ \hspace{0.2cm} \\
&  &  \\
\cline{1-1}\cline{4-4}
\multicolumn{1}{|c|}{ $\shortstack{{\ }\\$(X,D)$} $ } & & &  \multicolumn{1}{|c|}{%
$\mathcal{F}(X)$} \\
\multicolumn{1}{|c|}{$D(x,y)=\|{\delta}_y-{\delta}_x\|_{\mathcal{F}}$} & & {\Huge $\shortstack{{$\widehat{\delta}$}\\$\longrightarrow $} $} &  \multicolumn{1}{|c|}{
$\|Q\|_{\mathcal{F}}:=\underset{\substack{\phi\in L\\\|\phi\|_L\leq 1\\ \ }}{\mathrm{sup}}\langle\phi,Q\rangle$
} \\
\cline{1-1}\cline{4-4} &  &  \\
$\Big\downarrow $ & &    & $\Big\downarrow$ \\
& &  \\ \cline{1-1}\cline{4-4}
\multicolumn{1}{|c|}{} & & & \multicolumn{1}{|c|}{} \\
\multicolumn{1}{|c|}{$D_+(x,y)=\|{\delta}_y\!-\!{\delta}_x|_{\mathcal{F}_+}$} & $$ &  &
\multicolumn{1}{|c|}{$\|Q|_{\mathcal{F}_+}:=\underset{\substack{\phi\in L,\phi\geq 0\\\|\phi\|_L\leq 1\\ \ }}{\mathrm{sup}}\langle\phi,Q\rangle$} \\\cline{1-1}\cline{4-4}
\end{tabular}
\end{center}

\vspace{1cm}

\noindent Let us illustrate the above notion of canonical asymmetrization by means of the following simple example.

\begin{example}[Canonical asymmetrizations of $\mathbb{R}$] \label{sorg} Let us consider $\mathbb{R}$ as a metric space, with its usual distance $D(x,y)=|y-x|,$ for all $x,y,\in \mathbb{R}$ and $x_{0}=0$ as a
distinguished point. It is well-known (\cite{Godard,Weaver}) that the
free space $\mathcal{F}(\mathbb{R})$ can be identified with the space of
Lebesgue integrable functions $\mathcal{L}^{1}(\mathbb{R})$, provided we
identify the space $L=(\mathrm{Lip}_{0}(X,D),\Vert \cdot \Vert _{L})$ of
real-valued Lipschitz functions vanishing at $0$ with the Banach space $\mathcal{L}^{\infty }(\mathbb{R})$ (essentially bounded Lebesgue measurable functions) via the canonical linear isometry $T\phi =\phi ^{\prime }$ (a.e.), for all $\phi \in L$ ({\em c.f.} Rademacher theorem). Then taken either
\begin{equation*}
P=L_{+}=\{\phi \in L:\phi \geq 0\}\quad \text{or, respectively,\quad }
P=\{\phi \in L:\phi ^{\prime }\geq 0\},
\end{equation*}
leads to two different canonical asymmetrizations of $\mathbb{R}$ (via the
asymmetrizations $\Vert \cdot |_{\mathcal{F}_{+}}$ and respectively $\Vert
\cdot |_{\mathcal{F}_{P}}$ of its free space). The first asymmetrization is
given by the formula
\begin{equation*}
D_{+}(x,y)=\Vert {\delta}(y)-{\delta}(x)|_{\mathcal{F}_{+}}=\sup
_{\substack{ \phi \in L_+ \\ \Vert \phi \Vert _{L}\leq 1}}\left( \phi
(y)-\phi (x)\right) .
\end{equation*}
Notice that $D_{+}(x,y)\leq \max \{|y-x|,|y|\}$. It can be easily seen that
if either $y>x>0$ or $y<x<0,$ then $D_{+}(x,y)=|y-x|$ (take $\phi_{\ast }(t)=|t|$
in $L_{+}$ with $||\phi_{\ast }||_{L}=1$). However, $D_{+}(1,n)=n-1,$ while $D_{+}(n,1)=1$ for every $n\geq 2.$ \smallskip

\noindent The second asymmetrization, thanks to the monotonicity of every $\phi$ in $P$, yields that for all $x,y\in X$
\begin{align*}
D_{P}(x,y)=\Vert {\delta}_y-{\delta}_x|_{\mathcal{F}_{P}}&=\sup_{\substack{\phi \in L,\,\phi ^{\prime }\geq 0 \\ \Vert \phi \Vert _{L}\leq 1}}\left(
\phi (y)-\phi (x)\right) \\ &=\max \{y-x,\,0\}=u(y-x)=d_{u}(x,y),
\end{align*}
where $u(\cdot )$ is the asymmetric hemi-norm given by $u(x)=\max\{x,0\}$ for all $x\in \mathbb{R}$ and $d_{u}$ the
corresponding quasi-hemi-distance.
\end{example}
\smallskip
\subsection{Symmetrized distance and topologies}
Every quasi-metric distance can be symmetrized in the sense of the following definition.

\begin{definition}[Symmetrized distance] \label{sym-dis}Let $(X,d)$ be a quasi-metric space. Then
\begin{equation}\label{eq-dist}
d^{s_0}(x,y)=\max\{d(x,y),d(y,x)\}\qquad\mbox{and}\qquad d^s(x,y)=d(x,y)+d(y,x)
\end{equation}
are two natural symmetrizations of the quasi-distance $d$, which are equivalent to each other:
\[
d^{s_0}(x,y)\leq d^{s}(x,y)\leq 2\,d^{s_0}(x,y), \quad\text{for all } x,y\in X.
\]
\end{definition}

\noindent If $d$ is an extended quasi-metric, then so is $\bar{d}$ and consequently the symmetrizations $d^s$ and $d^{s_0}$ give rise to extended metrics. In case that $X$ is a vector space and $d$ satisfies \eqref{invariant}, the above symmetrizations preserve the invariance by translations and homothety. Notice further that \eqref{eq:equiv} shows that the symmetrization of the $P$-asymmetrized norm $\|\cdot|_{\mathcal{F}_P}$ of a free space $\mathcal{F}(X)$ is equivalent to $\|\cdot\|_{\mathcal{F}}$ ({\it c.f.} Remark~\ref{naturalasymmetrizednorm}). A similar remark applies to the symmetrization of the $P$-asymmetrization of the distance of a metric space $(X,D)$ ({\it c.f.} Definition~\ref{canoasym}).

\begin{proposition}[Asymmetrization vs symmetrization]
 Assume that $(X,D_P)$ is a $P$-asymmetrization of a metric space $(X,D)$ (\textit{c.f.} Definition~\ref{canoasym}).  Then the symmetrizations $D_P^s$ and $D_P^{s_0}$ are bi-Lipschitz equivalent to the initial distance $D$, and consequently, the Banach spaces $\Lip_0(X,D)$, $\Lip_0(X,D_P^s)$ and $\Lip_0(X,D_P^{s_0})$ are isomorphic.
\end{proposition}
\begin{proof}
It suffices to prove the result for $D_P^{s}$. Take $x,y\in X$. Let $\hat{\phi}$ be a function in $L=\Lip_0(X,D)$ with $\|\hat{\phi}\|_L\leq 1$ such that
$$D(x,y)=\sup_{\substack{\phi\in L\\\|\phi\|_L\leq 1}}(\phi(y)-\phi(x))=\hat{\phi}(y)-\hat{\phi}(x).$$
Let ${\hat{\phi}}_1$ and ${\hat{\phi}}_2$ be functions in $P$ satisfying $\hat{\phi}=\hat{\phi}_1-\hat{\phi}_2$,
with the inequality\linebreak $\max\{\|\hat{\phi}_1\|_L,\|\hat{\phi}_2\|_L\}\leq \|\hat{\phi}\|_L=1$. Then
\begin{align*}
D(x,y)&=(\hat{\phi}_1(y)-\hat{\phi}_1(x))+(\hat{\phi}_2(x)-\hat{\phi}_2(y))\\ &\leq \sup_{\substack{\psi\in P\\\|\psi\|_L\leq 1}}(\psi(y)-\psi(x))+\sup_{\substack{\psi\in P\\\|\psi\|_L\leq 1}}(\psi(x)-\psi(y)),
\end{align*}
which coincides with $D_P(x,y)+D_P(y,x)=D_P^s(x,y).$
Furthermore, it is clear that $$D_P^s(x,y)=D_P(x,y)+D_P(y,x)\leq 2D(x,y).$$
Thus, the distances $D_P^s$ and $D$ are equivalent, and $\Lip_0(X,D)$ is linear isomorphic to $\Lip_0(X,D_P^s)$.
\end{proof}
\medskip

\noindent Every (possibly extended) quasi-metric space $(X,d)$ can be endowed with three ``natural'' topologies:\smallskip
\begin{enumerate}
\renewcommand\labelenumi{(\roman{enumi})}
\leftskip .35pc
\item The \emph{forward topology} $\mathcal{T}(d)$, generated by the family of open \emph{forward}-balls
$$\{B_{d}(x,r)\hbox{\rm :}\ x\in X, r>0\},$$
where $B_{d}(x,r):=\{y\in X\hbox{\rm :}\ d(x,y)<r\}, \, \text{for all } x\in X\,\text{ and } r>0.$
\smallskip

\item The \emph{backward topology} $\mathcal{T}(\bar{d})$, generated by the family of \emph{backward}-balls
    $$\{B_{\bar{d}}(x,r)\hbox{\rm :}\ x\in X, r>0\},$$
where $B_{\bar{d}}(x,r):=\{y\in X\hbox{\rm :}\ d(y,x)<r\}, \, \text{for all } x\in X\,\text{ and } r>0.$ \smallskip

\item The \emph{symmetric topology} $\mathcal{T}^s$,  generated by the family of sets
$$\{B_{d}(x,r)\cap B_{\hat{d}}(x,r)\hbox{\rm :}\ x\in X,r>0\}.$$
\end{enumerate}

\noindent The symmetric topology being generated by the symmetrized distance $d^{s_0}$ or $d^s$ defined in \eqref{eq-dist} is obviously a metric topology. On the other hand, $\mathcal{T}(d)$ and $\mathcal{T}(\bar{d})$ are not in principle metric topologies. Nevertheless, they are both first countable topologies, since they have local bases consisting of balls of rational radii.\smallskip

\noindent In what follows, unless stated otherwise, the default topology on a quasi-metric space $(X,d)$ will be its forward topology, which is either a $T_1$-topology (when $d$ is a quasi-metric) or a $T_0$-topology (when $d$ is a quasi-hemi-metric). \smallskip

\begin{example}[$(\R,d_u)$] Let us consider $\R$ with its (canonical) asymmetric distance $d_u$ (see Example~\ref{sorg}). It is easy to check that $\mathcal{T}(d_u)$ has a local basis of the form $\{[x_0,x_0+\varepsilon):\;\varepsilon>0\}$ for each $x_0\in \mathbb{R}$, while $\mathcal{T}(\bar{d_u})$ has a local basis consisting of sets of the form $(x_0-\varepsilon,x_0]$, and $\mathcal{T}(d_u^s)$ is the usual topology of $\mathbb{R}$.\smallskip \\
Notice that $d_u$ is issued from the asymmetric hemi-norm $u(x)=\max\{x,0\}$ for all $x\in\mathbb{R}$, see~\eqref{u} and~\eqref{qnqd}. Moreover, the unit ball $\overline{B}(0,1)=\{y\in\mathbb{R}:d_u(0,y)\leq 1\}=(-\infty,1]$ is not $\mathcal{T}(d_u)$-closed because $(1,\infty)$ is not $\mathcal{T}(d_u)$-open. Notice also that, for every topological space $X$, a function $f:X\to\mathbb{R}$ is upper semicontinuous if and only if $f:X\to(\mathbb{R},u)$ is continuous.
\end{example}
\smallskip

The following example reveals that the topology of a quasi-metric space, which is $T_1$, may not be $T_2$.
\begin{example}
Let $\{x_n\}_{n\in \mathbb{N}}$ be a sequence of distinct elements and consider the space $$X=\{x_n\;:\,n\in\mathbb{N}\}\cup\{\bar x, \,\bar y\},$$ where $\bar x$ and $\bar y$ are different from each other and from any element of the sequence. Then the function $d$ defined on $X\times X$ by $d(\bar x,\,x_n)=d(\bar y,\, x_n)=1/n$, for every $n\in \mathbb{N}$, and $d(x,y)=1$ for all other cases where $x\neq y$, is a quasi-metric on $X$. In this case, the forward topology $\mathcal{T}(d)$ cannot be $T_2$, since $\{x_n\}_n$ converges to both $\bar x$ and $\bar y$. Notice that the symmetrized distance $d^s$ is discrete, with $d^s(x,y)>1$, whenever $x\neq y$.
\end{example}
\smallskip
\subsection{Cones and conic norms}  In this subsection we shall recall from \cite{valero2006quotient} the notion of an abstract cone. To this end, let us first recall that a \emph{monoid} is a semigroup $(X,+)$ with neutral element $0$.

\begin{definition}[Abstract cone] \label{defcone} A \emph{cone} on $\mathbb{R}_+$ is a triple $(C,+,\cdot)$ such that $(C,+)$ is an abelian monoid (with neutral element $0$), and
$\cdot $ is a mapping from $\mathbb{R}_{+}\times X$ to $X$ such that
for all $x,y\in C$ and $r,s\in \mathbb{R}_{+}$:
\begin{enumerate}
\renewcommand\labelenumi{(\roman{enumi})}
\leftskip .35pc
\item $r\cdot (s\cdot x)=(rs)\cdot x$;

\item $r\cdot (x+y)=(r\cdot x)+(r\cdot y) \quad\text{and}\quad (r+s)\cdot x=(r\cdot x)+(s\cdot x)$;

\item $1\cdot x=x \quad\text{and}\quad  0\cdot x=0.$\vspace{-.4pc}
\end{enumerate}
\end{definition}
Note that this definition does not include the existence of additive inverses. However, when such an inverse exists for some $x\in C$, it is unique, and we denote it by $-x$.

\medskip
\noindent A \emph{subcone} of a cone $(C,+,\cdot)$ is a cone $(S,+|_{S},\cdot|_{S}) $ such that $S$ is a subset of $C$ and
$+|_{S}$ and $\cdot|_{S}$ are, respectively, the restriction of $+$ and $\cdot$ to $S\times S$.

\begin{definition}[Cancellative cone] A cone $(C,+,\cdot)$ is called \emph{cancellative} if for any
$x,y,z\in C $, $$x+z=y+z \implies x=y.$$
\end{definition}
It follows readily that every cone that embeds to a linear
space is cancellative. Before we proceed, let us give two examples of abstract cones which are not cancellative.

\begin{example}[Non-cancellative cone]  (i). Consider a cone $C$ and let $\mathcal{S}(C)$ be the set of subcones of~$C$, under the usual operations of subset addition and scalar product. Then $\mathcal{S}(C)$ may not be cancellative.  \\
Indeed, for $C=\mathbb{R}^2$, let us consider the following elements of $S(C)$:
$$X=\{(x,0)\,:\,x\in \mathbb{R}\}, \quad Y=\{(0,x)\,:\,x\in \mathbb{R}\}\quad \text{and} \quad Z=\{(x,x)\,:\,x\in\mathbb{R}\}.$$
It follows that $X+Z=Y+ Z$ but  $X\neq Y.$  \smallskip \\
\noindent (ii). For a nonempty set $X$, consider the set of non-negative functions $\mathbb{R}_+^X$, with the operations
$\lambda \odot f=f^\lambda$ (product with external scalar) and $f\oplus g=f\cdot g$ (addition). Then $\mathbb{R}_+^X$ is not cancellative.
\end{example}

\smallskip
\begin{definition}[Cone morphisms]\label{cone-linear}
 A \emph{linear mapping} from a cone
$(C_1,+,\cdot)$ to a cone $(C_2,+,\cdot)$ is a mapping $f\hbox{\rm :}\
C_1\to C_2$ such that $f(\alpha \cdot x+\beta \cdot y)=\alpha
\cdot f(x)+\beta \cdot f(y)$ for any $x,y\in C_1$ and any $\alpha,\beta \in \mathbb{R}_+$.
\end{definition}
\smallskip

\begin{remark}[Compatibility of cone morphisms]
Let $f$ be a linear mapping between two cones $C_1$ and $C_2$. Then if $H_i:=\{x\in C_i: \,-x\in C_i \}$ denotes the linear part of the cone $C_i$, for $i\in\{1,2\}$, then it is straightforward to see that for every $x\in H_1$,
$f(-x)=-f(x)$. In particular, the restriction of $f$ onto $H_1$ yields a linear mapping between the linear spaces $H_1$ and $H_2$.
\end{remark}

We shall now introduce the notion of a {\em conic-norm}, which will be relevant for our developments.

\begin{definition}[Conic norm] \label{defnorm}
A {\em conic-norm} on a cone $(C,+,\cdot)$ is a function $\|\cdot|\hbox{\rm :}\ C\to \mathbb{R}_{+}$ such that for all
$x,y\in C$ and $r>0$:\smallskip
\begin{enumerate}
\renewcommand\labelenumi{(\roman{enumi})}
\leftskip .35pc
\item $\|x+y|\leq \|x|+\|y|$;\smallskip
\item $\|x|=0\ \iff\ x=0$; \smallskip
\item $\|r\cdot x|=r\|x|$. \vspace{-.4pc}\smallskip
\end{enumerate}
\medskip
The pair $(C,\|\cdot|)$ is called \emph{normed cone}. If we replace condition (ii) by
$$
{\rm (ii)' }\qquad  x=0\iff \forall z\in C,\, \left[ x+z=0 \implies \|x|=\|z|=0\right],
$$
then we say that $\|\cdot|\hbox{\rm :}\ C\to \mathbb{R}_{+}$ is a {\em conic hemi-norm}. A cone equipped with either a conic-norm or a conic hemi-norm will be called {\em normed cone}. This is in accordance with the terminology {\em asymmetric normed space}, which refers to a vector space equipped with either an asymmetric norm or an asymmetric hemi-norm. (The asymmetry is now stemming from the use of a cone, rather than a vector space. Notice however that $C$ is not necessarily a cancellative cone.)
\end{definition}

\begin{example}\label{exfigure}
Consider the pair $(\mathbb{R}^2,\|\cdot|$), with
$$\|(x_1,x_2)|:=u(x_1)+u(x_2),$$
where $u$ is the canonical asymmetric hemi-norm of $\R$ given by
$u(x)=\max\{x,0\}$ for all $x\in \mathbb{R}$ (see also Example~\ref{sorg}).
By restricting $\|\cdot|$ to any cone $C\subseteq\mathbb{R}^2$, we obtain a conic-hemi-norm. The case $C=\mathbb{R}_{-}^2$ corresponds to an example of normed cone with the trivial conic hemi-norm equal to $0$ everywhere.
\begin{figure}[h]
\centering
\includegraphics[width=7.1cm]{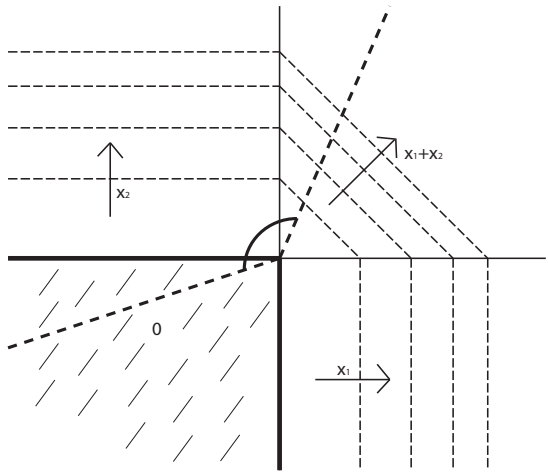}
\caption{Illustration of Example \ref{exfigure}}\label{fig}
\end{figure}

\end{example}

\begin{remark}[Terminology alert II]
 The reader should again be alerted that some authors (\cite{valero2006quotient}, \emph{e.g.}) employ the term of \emph{quasi-norm} to refer to what we call ``conic hemi-norm''. We decided to opt for the term ``conic hemi-norm'' because it is more suggestive. At the same time, the term ``quasi-norm'' might have a different meaning in the theory of  Banach spaces (\cite{AK2009}, {\em e.g.}). The asymmetric aspect of the conic-norm is inherent to the definition of a cone, and therefore does not require the prefix ``quasi''.
\end{remark}

\begin{remark}[Conic-norm vs asymmetric norm]  If the cone happens to be a linear space $X$, then the conic-norm corresponds to an asymmetric norm on $X$, and instead of the term ``normed cone'' we use the term \emph{asymmetric normed space}, as in \cite{cobzas2012functional}. The same applies to the case of conic hemi-norms and asymmetric hemi-norms. Given an asymmetric normed space $(X,\|\cdot|)$, one can define the \emph{reverse} norm of an element $x\in X$ as $\|\!\!-\!\!x|$, and the (symmetric) norms (symmetrizations of $\|\!\cdot|$)
$$\|x\|_{s_0}:=\max\{\|x|,\|\!\!-\!\!x|\} \qquad \text{and}\qquad \|x\|_{s}:=\|x| + \|\!-\!x|.$$

\noindent It is clear that the above norms are equivalent. 
\end{remark}

\noindent An extended quasi-metric $d$ on a cone $(C,+,\cdot)$ is called \emph{invariant} if it satisfies
\begin{equation}\label{inv2}
d(x+z,y+z)=d(x,y)\quad\text{and}\quad d(rx,ry)=rd(x,y),
\end{equation}
which is the case whenever the extended quasi-metric $d$ is induced by a conic-norm which is the restriction of an asymmetric norm of a vector space that contains $C$. An extended quasi-metric $d$ on a cone $(C,+,\cdot)$ is called \emph{subinvariant} if $d(x+z,y+z)\leq d(x,y)$ instead of the first part of \eqref{inv2}. More generally, the following result, established in \cite[Proposition 1]{garcia2004metrizability}, states that given a normed cone $(C,\|\cdot|)$, there is a natural way to generate an extended quasi-metric $d_e$.
\begin{proposition}[Extended quasi-metrics generated by conic-norms]\label{extendedqm}
\noindent Let $\|\cdot|$ be a conic-(hemi-)norm on a cone
$(C,+,\cdot)$. Then the function $d_e$ defined on $C\times C$ by
\begin{equation*}
d_e(x,y) = \inf_{\substack{z\in C\\y=x+z}} \|z|,
\end{equation*}
is a subinvariant extended quasi(-hemi)-metric on $C$. If the cone $(C,+,\cdot)$ is cancellative, then $d_e$ is invariant.
\end{proposition}
For $x\in C,$ $r\in \mathbb{R}_{+}\backslash \{0\}$ and $\varepsilon >0$, we have
$$rB_{d_e}(x,\varepsilon)=rx+\{y\in C\hbox{\rm :}\ \|y|<r\varepsilon\},$$ and the translations are $\mathcal{T}(d_e)$-open.

\begin{remark}\label{excl} {\rm (i)} The quasi-metric $d_e$ might take infinite values as long as $C$ is not a linear space (the infimum may be taken over the empty set). \smallskip\\
{\rm (ii)} If $C$ is a cancellative cone, then the infimum in the above definition becomes superfluous, and if $C$ is a linear space, the definition of $d_e$ coincides with the definition of the quasi-metric given in \eqref{qnqd}. \smallskip\\
{\rm (iii)} The quasi-metric induced by the reverse norm coincides with the one obtained by the reverse quasi-metric. The same is true for the symmetrized metric which coincides with the metric obtained by the symmetrization of the asymmetric norm.
\end{remark}

\noindent Using the extended quasi-metric of Definition~\ref{extendedqm}, we define an equivalence between normed cones.
\begin{definition}[Isomorphisms between normed cones]\label{defisocone}
A bijective mapping $\Phi:X\to Y$ between two normed cones is called an \emph{isometric isomorphism} if it is linear ({\em c.f.} Definition~\ref{cone-linear}) and an isometry between the corresponding extended quasi-metrics, that is, $$d_e(\Phi x_1,\Phi x_2)=d_e(x_1,x_2)\,,\quad\text{for all } x_1,x_2\in X.$$
Note that this is equivalent to the relation $\|\Phi x|=\|x|$, for all $x\in X.$
\end{definition}
We shall now proceed to define a notion of {\em completeness} for a quasi-metric space. Even though there are several non-equivalent notions of completeness in quasi-metric spaces (all of them generalizing, in some sense, completeness in metric spaces), we will focus on the one which is compatible with normed cones and asymmetric normed spaces:

\begin{definition}[Bicomplete quasi-metric space] A (possibly extended) quasi-metric space $(X,d)$ is called \emph{bicomplete} if the (extended) metric space $(X,d^s)$ is complete, meaning that any $d^s$-Cauchy sequence in $X$ is $d^s$-convergent in $X$. If $X$ is a linear space and $d$ is the quasi-metric induced by an asymmetric norm $\|\cdot|$, we say $(X,\|\cdot|)$ is a bi-Banach space, whenever $X$ is complete under the symmetrized metric~$d^s$.
\end{definition}
\medskip
\begin{definition}[Bicompletion of a quasi-metric space] Let $(X,d)$ be an (extended) quasi-metric space. A \emph{bicompletion} of $(X,d)$ is an (extended) quasi-metric space $(\tilde{X},\tilde{d})$, along with a mapping
$$\iota:(X,d)\to (\tilde{X},\tilde{d})$$ such that:
\begin{enumerate}\renewcommand\labelenumi{(\roman{enumi})}
\leftskip .35pc
    \item $\iota$ is an isometric embedding;\smallskip
    \item $\iota(X)$ is dense in $\tilde{X}$ for the symmetrized topology;\smallskip
    \item $\left(\tilde{X},\tilde{d}\right)$ is bicomplete.
\end{enumerate}
\end{definition}
\medskip
An important result regarding bicompleteness of normed cones (and therefore of asymmetric normed spaces) is the existence and uniqueness of the bicompletion, see \cite[Theorem 3.13]{oltra2004isometries}. This result, once again, generalizes the usual completion of normed linear spaces.

\begin{proposition}[Uniqueness of bicompletion for cancellative normed cones]
 Let $(C,\|\cdot|)$ be a cancellative normed cone. Then there exists a unique (up to an isometric isomorphism) bicompletion of $(C,\|\cdot|)$, which is also a normed cone, and the embedding into the bicompletion is linear. If $C$ is a linear space, then its bicompletion is an asymmetric normed space.
\end{proposition}

\subsection{Semi-Lipschitz functions and dual cones}

Let us now define the class of semi-Lipschitz functions, which reflects naturally the asymmetry in the definition of a quasi-metric space.
\begin{definition}[Semi-Lipschitz function]\label{slip}
Let $(X,d)$ be a quasi-metric space. A function $f:X\to \mathbb{R}$ is said to be \emph{semi-Lipschitz} if there exists $L>0$ such that for every $x,y\in X$ we have:
\begin{equation}\label{SL}
f(x)-f(y)\leq L\,d(y,x).
\end{equation}
The class of semi-Lipschitz functions on $X$ is denoted by $\SLip(X)$.
\end{definition}
Let us recall that a Lipschitz function $f$ satisfies $|f(x)-f(y)|\leq L\,d(x,y)$ for all $x,y\in X$. Therefore, if $(X,d)$ is a metric space, the notions of semi-Lipschitz and Lipschitz function coincide. In a quasi-metric space, $f$ is Lipschitz if and only if both $f$ and $-f$ are semi-Lipschitz. To get easy examples of semi-Lipschitz functions that are not Lipschitz, consider functions of the form $d(x,\cdot)$ on the quasi-metric space of Example~\ref{sorg}.

\begin{definition}[Semi-Lipschitz conic-norm] \label{def-SL-norm} Let $(X,d)$ be a quasi-metric space.
The \emph{semi-Lipschitz conic (hemi-)norm} of a function $f:X \to \mathbb{R}$ is defined by
$$\|f|_S:=\inf\big\{L>0: \,\,\eqref{SL} \text{ holds}\big\}.$$
\end{definition}

The following proposition holds easily.

\begin{proposition}[Semi-Lipschitz criterium]\label{SLcri}
Let $(X,d)$ be a quasi-metric space and $f:X\to \mathbb{R}$.
\begin{itemize}
\item[i)] If $d$ is a quasi-metric, then
$f$ is semi-Lipschitz if and only if
$$\|f|_S=\sup_{x\neq y}\frac{\max\{f(x)-f(y),0\}}{d(y,x)}=\sup_{x\neq y}\frac{f(x)-f(y)}{d(y,x)}<\infty.$$
\item[ii)] If $d$ is a quasi-hemi-metric, then
$f$ is semi-Lipschitz if and only if $\|f|_S<\infty$.
In this case,
$$\|f|_S=\sup_{d(y,x)>0}\frac{\max\{f(x)-f(y),0\}}{d(y,x)}=\sup_{d(y,x)>0}\frac{f(x)-f(y)}{d(y,x)}.$$
\end{itemize}
\end{proposition}

\begin{remark}
Let $(X,d)$ be a quasi-metric space and $f:X\to \mathbb{R}$. If for all $x,y\in X$ we have $f(x)\leq f(y)$ whenever $d(y,x)=0$ ($d$-monotonicity), then the following equality holds:
\begin{equation}\label{eqsupma}
\sup_{d(y,x)>0}\frac{\max\{f(x)-f(y),0\}}{d(y,x)}=\sup_{d(y,x)>0}\frac{f(x)-f(y)}{d(y,x)}\,.
\end{equation}
It follows readily from Definition~\ref{slip} that every semi-Lipschitz function is $d$-monotonic, and therefore it satisfies~\eqref{eqsupma}.
\end{remark}

\begin{examples}
(i) If $f:X\to \mathbb{R}$ is not semi-Lipschitz or $d$-monotonic, then the equality \eqref{eqsupma} is not necessarily true. For example, let $X=\{a,b\}$ with $a,b\in\mathbb{R}$, consider $d:X\times X\to[0,\infty)$ the quasi-hemi-metric given by $d(a,b)=1$ and $d(b,a)=0$, and let $f:X\to\mathbb{R}$ defined as $f(a)=1$ and $f(b)=0$. Then $f$ is not semi-Lipschitz, $\displaystyle{\sup_{d(y,x)>0}\frac{f(x)-f(y)}{d(y,x)}=-1}$ and $\displaystyle{\sup_{d(y,x)>0}\frac{\max\{f(x)-f(y),0\}}{d(y,x)}=0}$.

\noindent (ii) The equality \eqref{eqsupma} could be true without $f$ being semi-Lipschitz. For instance, let $X=\{a,b,c\}$ with $a,b,c\in\mathbb{R}$, consider $d:X\times X\to[0,\infty)$ the quasi-hemi-metric given by
$$d(x,y)=\begin{cases}
\,\,1,& \mbox{ if }x=a,y=b\\
\,\,1,& \mbox{ if }x=b,y=c\\
\,\,2,& \mbox{ if }x=a,y=c\\
\,\,0,& \mbox{ otherwise}
\end{cases},$$
and let $f:X\to\mathbb{R}$ defined as $f(a)=2$, $f(b)=1$ and $f(c)=1$. Then $f$ is not semi-Lipschitz, since $f(a)-f(b)=1$ and $d(b,a)=0$. However, $\displaystyle{\sup_{d(y,x)>0}\frac{f(x)-f(y)}{d(y,x)}=0}$ and $\displaystyle{\sup_{d(y,x)>0}\frac{\max\{f(x)-f(y),0\}}{d(y,x)}=0}$.
\end{examples}

\begin{remark}[Terminology alert III]\label{remarkim} The above definition of semi-Lipschitz function, introduced in~\cite{DJV}, differs from the one that is usually considered in the literature and is based on an inequality of the form:
\begin{equation}
    f(x)-f(y)\leq L\,d(x,y).\label{otroslip}
\end{equation}

A function $f:(X,d)\to \mathbb{R}$ is semi-Lipschitz according to Definition~\ref{slip} if and only if it is semi-Lipschitz on $(X,\bar{d})$ according to~\eqref{otroslip}. This is also equivalent to $-f$ being semi-Lipschitz on $(X,d)$ according to~\eqref{otroslip}. Therefore, the difference between these two definitions of a semi-Lipschitz function is equivalent to either a change of orientation of the quasi-metric (replace $d$ by $\bar d$) or of the sign of the values of $f$ (replace $f$ by $-f$). With this in mind, let us now justify our choice for Definition~\ref{slip}:
\medskip
\begin{enumerate}\renewcommand\labelenumi{(\roman{enumi})}
\leftskip .35pc
    \item If $(X,\|\cdot|)$ is a normed cone, the norm $\|\cdot|$ may not be semi-Lipschitz according to~\eqref{otroslip}, while $-\|\cdot|$ is always semi-Lipschitz according to~\eqref{otroslip}. \smallskip
    \item In general, if $(X,d)$ is a quasi-metric space, the functions of the form $d(x_0,\cdot)$ that characterize forward convergence (in the sense that $\{x_n\}_n\!\longrightarrow x_0$ in the forward topology if and only if $d(x_0,x_n)\!\longrightarrow 0$) may not be semi-Lipschitz according to~\eqref{otroslip}, while $-d(x_0,\cdot)$ and $d(\cdot,x_0)$ will~be~so.
\end{enumerate}
Therefore, to avoid/circumvent the above inconveniences, we shall opt for Definition~\ref{slip}. This definition, in particular, is compatible with the natural definition of a semi-Lipschitz function from a quasi-metric space $(X,d)$ towards an asymmetric normed space or normed cone $(Y,\|\cdot|)$: indeed, denoting by $d^{\|\cdot|}$ the distance associated to the asymmetric norm $\|\cdot|$, it is natural to ask
$$d^{\|\cdot|}(f(x),f(y))=\|f(y)-f(x)|\leq L\,d(x,y),$$
which coincides with our definition when taking $(Y,\|\cdot|)=(\mathbb{R},u)$, with $u$ given by $u(x)=\max\{x,0\}$ for all $x\in \mathbb{R}$. In fact, the quasi-metric space $(\mathbb{R},u)$ is involved in the definition of the dual of both an asymmetric normed space and a normed cone, and consequently, it is of great importance in this theory. Furthermore, as we shall see in Proposition~\ref{caract}, a real valued linear functional on a normed cone will belong to the {\em dual cone} (see forthcoming Definition~\ref{dual cone}) if and only if it is semi-Lipschitz according to Definition~\ref{slip}.
\end{remark}

\medskip
\begin{definition}[Asymmetric pivot space] \label{Aps} Let $(X,d)$ be a quasi-metric space and $x_0\in X$ be a base point. We define the asymmetric nonlinear dual (pivot space)
$$\mathrm{SLip}_0(X,d):=\{f\in \mathrm{SLip}(X)\text{ such that }f(x_0)=0\}.$$
In case there is no ambiguity regarding the considered quasi-metric, we simply write $\mathrm{SLip}_0(X)$.
\end{definition}
\smallskip

\begin{remark}\label{lip} \
(i). It is easy to see that $\left(\mathrm{SLip}_0(X),\|\cdot|_S\right)$ is a cancellative normed cone.\smallskip\\
(ii). Any semi-Lipschitz function on a quasi-metric space $(X,d)$ is Lipschitz on the (symmetrized) metric space $(X,D)$, where $D$ is either $d^s$ or $d^{s_0}$. Therefore, both cones of semi-Lipschitz functions $\mathrm{SLip}(X,d)$ and $\mathrm{SLip}(X,\bar{d})$ are contained in the linear space  $\mathrm{Lip}(X,D)$ of Lipschitz functions on $(X,D)$.
\end{remark}

Let $\left(\mathbb{R}, u\right)$ be the asymmetric normed space evoked in Example~\ref{sorg}. Then the asymmetric norm $u$ generates the upper topology on $\mathbb{R}$, which is the topology that characterizes upper semicontinuity in the following way: a function from a topological space $f:(X,\tau)\to (\mathbb{R},u)$ is continuous for the forward topology of $\left(\mathbb{R},u\right)$ if and only if $f$ is upper semicontinuous for the usual norm on $\mathbb{R}$ (which is the symmetrization of $u$).

\begin{example}\label{dist}
Let $(X,d)$ be a quasi-metric space with base point $x_0$. Then for each $x\in X$, the function $f(\cdot)=d(x,\cdot)-d(x,x_0)$ belongs to $\mathrm{SLip}_0(X,d)$ and satisfies $\|f|_S = 1$.
Indeed, it follows directly from the triangular inequality that $f$ is semi-Lipschitz with $\|f|_S\leq 1$. We obviously have $f(x_0)=0$. Taking $z\in X$ with $z\neq x$ we deduce $f(z)-f(x)=d(x,z)$ that is $\|f|_S=1$.
\end{example}

The previous example becomes relevant in order to define duality for normed cones and asymmetric normed spaces. The following proposition gives some insight for this duality. The proof has no essential difficulty and is included for the reader's convenience.

\begin{proposition}[Linear functionals over a normed cone]\label{caract}
Let $(C,\|\cdot|)$ be a normed cone and $\varphi:C\to \mathbb{R}$ a linear functional. Then the following are equivalent:
\begin{enumerate}\renewcommand\labelenumi{(\roman{enumi})}
\leftskip .35pc
    \item[{\rm (i)}] $\varphi$ is upper semicontinuous (in short, usc); \smallskip
    \item[{\rm (ii)}] $\varphi$ belongs to $\mathrm{SLip}_0(C,d_e)$, where $d_e$ is the (extended) quasi-metric induced by the conic-norm~$\|\cdot|$ (c.f. Proposition~\ref{extendedqm});\smallskip
    \item[{\rm (iii)}] there exists $M\geq 0$ such that $\varphi(x)\leq M\|x|$, for all $x\in C$.
\end{enumerate}
\end{proposition}

\begin{proof}
Let us show that (i) implies (iii). Assume that the linear functional $\varphi$ is usc. Then there exists  $\alpha>0$ such that $\varphi\left(B(0,\alpha)\right)\subseteq (-\infty,1)$. Set $M=2/\alpha$. Then for every $x\in C$ with $\|x|\neq 0$, we have $\tilde{x}=\frac{\alpha x}{2\|x|}\in B(0,\alpha)$, hence $\varphi\left(\tilde{x}\right)<1$ and $\varphi(x)< M\|x|$. If $x\in C$ with $\|x|=0$, then for every $r>0$ we have $\|rx|=0$ and $\varphi(rx)<1$, which implies $\varphi(x)<\frac{1 }{ r}$ and
necessarily $\varphi(x)\leq 0$. \smallskip \\ 
Let us now show that (iii) implies (ii). We need to establish the inequality $\varphi(x)-\varphi(y)\leq L\, d_e(y,x),$ $\forall x,y\in C$, for some $L\geq 0$. If $d_e(y,x)=\infty$, the inequality becomes trivial. If not, then $x\in y+C$, so we can write $x=y+z$, and then $\varphi(x)-\varphi(y)=\varphi(z)\leq M\|z|.$
By taking infimum of all $z$ such that $x=y+z$, we get that $\varphi(x)-\varphi(y)\leq M d_e(y,x)$, that is, $\varphi$ is semi-Lipschitz.\smallskip \\
Let us finally assume (ii) and recall that the forward topology on $(C,\|\cdot|)$ is first countable. Then take $\{x_n\}_{n}\subseteq C$ such that $d_e(x,x_n)\to 0$. Since $\varphi$ is semi-Lipschitz, we have $\varphi(x_n)-\varphi(x)\leq Ld_e(x,x_n)$ for some $L\geq 0$, which yields that $\varphi(x)\geq\limsup \varphi(x_n)$.
\end{proof}
\smallskip
\begin{remark}\label{rem-22}
Each one of the above statements is also equivalent to $\varphi$ being lower semicontinuous (in short, lsc) for the reverse extended quasi-metric $\bar{d_e}$:\smallskip\newline  Indeed, assume there exists $M\geq 0$ such that $\varphi(x)\leq M\|x|$ for all $x\in C$, and consider a sequence $\{z_n\}_n$ and $z$ in $C$ such that $\bar{d}_e(z,z_n)\to 0$. Then $d_e(z_n,z)\to 0$, which yields the existence of a sequence $\{y_n\}_n\subset C$ such that $y_n+z_n=z$ and $\|y_n|\to 0$. Since $\varphi$ is linear, $\varphi(z)=\varphi(z_n)+\varphi(y_n)\leq \varphi(z_n) + M\|y_n|$, which yields that $\varphi$ is lsc for $\bar{d_e}$. \smallskip\newline On the other hand, if $\varphi$ is lsc for $\bar{d}_e$, an analogous argument to Proposition~\ref{caract} ((i)$\implies$(iii)) leads to the same conclusion, that is, the existence of $M\geq 0$ such that $\varphi(x)\leq M\|x|$ for all $x\in C$.
\end{remark}
\smallskip
\begin{definition}[Dual normed cone]\label{dual cone}
Let $(C,\|\cdot|)$ be a normed cone. We define the \emph{dual cone} of $C$ as
$$C^*:=\{\varphi:C\to \mathbb{R}:\,\varphi\text{ usc, linear}\}\,=\,\{\varphi\in\SLip_0(C):\,\varphi \text{ linear}\}.$$
For any $\varphi\in C^*$, the \emph{dual conic-norm} is defined by
\begin{equation*}\label{dual-conic}
\|\varphi|^*:=\sup_{\|x|\leq 1} \max\{\varphi(x),0\}=\sup_{\|x|\leq 1} \varphi(x).
\end{equation*}
\end{definition}

\noindent It is easy to check that $\|\cdot|^*$ is a conic-norm on $C^*$ (obviously $\|\varphi|^*\geq 0$, since $\varphi(0)=0$).
Moreover, if $(C,\|\cdot|)$ is a normed cone with conic-hemi-norm, then $\|\cdot|^*$ is a conic-hemi-norm on $C^*$. \medskip

The proof of the following result is reasonably simple.

\begin{proposition}\label{npro}
Let $(C,\|\cdot|)$ be a normed cone, and $\varphi\in C^*$. Then
$$\|\varphi|^*=\inf\{M>0:\varphi(x)\leq M\|x|,\mbox{ for all }x\in C\}.$$
\end{proposition}

As in the case of normed spaces, there is a direct relation between the semi-Lipschitz constant and the dual norm of a linear functional:
\begin{corollary}[Dual conic-norm and semi-Lipschitz constant] Let\\
$(C,\|\cdot|)$ be a normed cone, and $\varphi\in C^*$. Then $\|\varphi |^*=\|\varphi|_S$ and the subcone of linear functionals of $\mathrm{SLip}_0(C)$ (linear semi-Lipschitz functions) is isometrically isomorphic to $\left(C^*,\|\cdot|^*\right)$ (linear usc functions).
\end{corollary}
\begin{proof}
The inequality $\|\varphi|_S\leq\|\varphi|^*$ follows from Proposition~\ref{caract} (see (ii)$\Rightarrow$(iii)). For the opposite inequality, since $\varphi$ is semi-Lipschitz and $\varphi(0)=0$ we get:
$$\varphi(x)=\varphi(x)-\varphi(0)\,\leq \,\|\varphi|_S\,d_e(0,x)=\|\varphi|_S\,\|x|,$$
yielding by Proposition \ref{npro} that $\|\varphi|^* \leq \|\varphi|_S.$
The proof is complete.
\end{proof}
\smallskip
\subsection{Duality of asymmetric normed spaces}\medskip

In this subsection we consider the particular case that the normed cone is an asymmetric normed space $(X,\|\cdot|)$. \smallskip

\begin{proposition}[Dual of a finite-dimensional space]
Let $(X,\|\cdot|)$ be an asymmetric normed space of finite dimension. Then there exists $M>0$ such that
\begin{equation}\label{infdim}
\|\!\!-\!\!x|\leq M\|x|,\ \mbox{for all }x\in X.
\end{equation}
Furthermore, $(X,\|\cdot|)^*$ is also an asymmetric normed space satisfying that for every $\varphi\in (X,\|\cdot|)^*$, $-\varphi\in (X,\|\cdot|)^*$ and $\|-\varphi|^*\leq M\|\varphi|^*$.
In particular, $(X^*,\|\cdot|)$ is a linear space (not only a normed cone).
\end{proposition}
\begin{proof}
Let $B=\{x\in X:\|x|\leq 1\}$ be the unit ball of $X$. Since in finite dimensions all asymmetric norms inducing a $T_1$-topology are equivalent (see \cite[Corollary~11]{GR2005} or \cite[Theorem~3]{BF2020} for example), it follows that $B$ is closed convex and $0\in\operatorname{int}B$. Thus we can assure the existence of $M>0$ such that $\left\|\frac{-x }{ \|x|}\right|\leq M$, for all $x\in X$ with $\|x|\neq 0$, which yields
$\|\!\!-\!\!x|\leq M\|x|$, for all $x\in X$. Now, if $\varphi\in (X,\|\cdot|)^*$ then
$$-\varphi(x)=\varphi(-x)\leq \|\varphi|^*\|\!\!-\!\!x|\leq M\|\varphi|^*\|x|,\ \mbox{for all }x\in X$$
and
$$\|\!\!-\!\!\varphi|^*\left(\ =\sup_{\|x|\leq1}-\varphi(x)\ \right)\leq M\|\varphi|^*,\ \mbox{for all }\varphi\in (X,\|\cdot|)^*.$$
The proof is complete.
\end{proof}

\begin{remark}[An infinite dimensional counterexample]
If $X$ is infinite-dimensional, then \eqref{infdim} may not be fulfilled. For example, let
$$X=\{f\in\mathcal{C}([0,1]):\int_0^1f(t)\,dt=0\}$$ and $\|f|:=\max_{t\in[0,1]}\max\{f(t),0\}$. 
Consider the sequence of functions $\{f_n\}_n\subset X$ defined as
$$f_n(x)=\begin{cases}
\qquad\frac{1 }{n}\,,\smallskip&\ \mbox{if }\ 0\leq x<\frac{1}{ n^2}\\
\frac{n}{2-n^2}x+\frac{1-n^2}{2n-n^3}\,,\smallskip&\ \mbox{if }\ \frac{ 1}{ n^2}\leq x<1-\frac{ 1}{ n^2}\\
-n^3x-n(1-n^2)\,,\smallskip&\ \mbox{if }\ 1-\frac{ 1}{ n^2}\leq x\leq 1\,  \\
\end{cases}\ \ (n\in\mathbb{N}).$$
Then $\|f_n|=1/n$ for each $n\geq 2$ and $\|-f_n|=n\longrightarrow \infty$, which contradicts \eqref{infdim}. \smallskip

\noindent In addition, $X^*$ is a normed cone (and not a vector space). To see this, let $\delta_1:\mathcal{C}([0,1])\to\mathbb{R}$ be defined as $\delta_1(f)=f(1)$. Then $\{f_n\}_n\!\longrightarrow 0$, $\delta_1(f_n)=-n\longrightarrow -\infty$ and $\delta_1(0)=0$, which shows that the linear functional $\delta_1$ is not lower semicontinuous in $(X,\|\cdot|)$.
\end{remark}

\begin{remark}[Continuity of evaluation functionals] \label{remarkevacontin}
Let $(X,\|\cdot|)$ be an asymmetric normed space with dual $X^*$. For every $x\in X$, the evaluation functional
$\widehat{x}:X^*\to\mathbb{R}$ defined as $\widehat{x}(\varphi)=\varphi(x)$ is linear and $\|\cdot|^*$-continuous. Indeed, we have
$$\widehat{x}(\varphi)=\varphi(x)\leq \|\varphi|^*\|x| \quad\text{and}\quad -\widehat{x}(\varphi)=-\varphi(x)=\varphi(-x)\leq \|\varphi|^*\|\!\!-\!\!x|,$$
which yields $|\widehat{x}(\varphi)|\leq \max\{\|x|,\|\!\!-\!\!x|\}\|\varphi|^*$, {\it i.e.} $\widehat{x}$ is Lipschitz and thus continuous.
\end{remark}

\begin{lemma}[$(\mathcal{L}^1(\mathbb{R}),\|\cdot|_{1,+})^*=(\mathcal{L}^\infty_+(\mathbb{R}), \|\cdot\|_\infty)$]\label{duall1}
Let $\mathcal{L}^1(\mathbb{R})$ be endowed with the asymmetric norm $$\|f|_{1,+}:=\int_{\R} f^+d\lambda,$$
where $f^+(x)=\max\{f(x), 0\}$ and $\lambda$ denotes the Lebesgue measure. Then, the dual of $(\mathcal{L}^1(\mathbb{R}),\|\cdot|_{1,+})$ is isometrically isomorphic to $(\mathcal{L}^\infty_+(\mathbb{R}), \|\cdot\|_\infty)$, where $\mathcal{L}^\infty_+(\mathbb{R})$ denotes the cone of nonnegative functions in $\mathcal{L}^\infty (\mathbb{R})$.
\end{lemma}
\begin{proof}
The facts that $(\mathcal{L}^1(\mathbb{R}),\|\cdot|_{1,+})$ is an asymmetric normed space and $(\mathcal{L}^\infty_+(\mathbb{R}), \|\cdot\|_\infty)$ is a normed cone are straightforward. Take $\varphi\in (\mathcal{L}^1(\mathbb{R}),\|\cdot|_{1,+})^*$. Then
$\varphi:\mathcal{L}^1(\mathbb{R})\to \mathbb{R}$ is linear and $(\|\cdot|_{1,+}$-$u)$-continuous (see Example~\ref{sorg}). Then, by Remark~\ref{lip}, $\varphi$ is continuous for the symmetrized norms in both spaces, therefore
$$|\varphi(f)|\leq \|\varphi\|^* \max\Big\{ \| f^+|_{1,+},\, \,\|\!\!-\!\!f^+|_{1,+}\Big\} \leq \|\varphi\|^*\|f\|_1,$$
\noindent where $\|\cdot\|^*$ denotes the dual norm of the normed space $(\mathcal{L}^1(\mathbb{R}),(\|\cdot|_{1,+})^s)$ and $\|\cdot\|_1$ is the usual norm on $\mathcal{L}^1(\mathbb{R})$. It follows that $\varphi$ is $(\|\cdot\|_1$-$|\cdot|)$-continuous, and therefore there exists $g\in \mathcal{L}^\infty(\mathbb{R})$ such that $\varphi(f)=\int gfd\lambda$ for all $f\in \mathcal{L}^1(\mathbb{R})$. \smallskip

\noindent We claim that $g\geq 0$ almost everywhere:\smallskip \\
Indeed, suppose, towards a contradiction, that there exists a set $E$ of measure $0<\lambda (E)<\infty$ such that $g<0$ on $E$. Consider the sequence $f_n=-n\mathbf{1}_E$ (where $\mathbf{1}_E$ is the characteristic function of $E$), which clearly belongs to $\mathcal{L}^1(\mathbb{R})$. On the other hand, since $\|f_n|_{1,+}=0$ for all $n\in \mathbb{N}$, the function $f_n$ belongs to the unit ball of the asymmetric norm $\|\cdot|_{1,+}$. Then, as $n\longrightarrow +\infty$, we deduce
$$\varphi(f_n)=\int g f_n d\lambda=\int_{E^c}gf_nd\lambda +\int_E gf_nd\lambda = n\int_E (-g) d\lambda \longrightarrow +\infty.$$
Therefore, $\varphi$ can not be $(\|\cdot|_{1,+}$-$u$)-continuous, a contradiction.\medskip \\

Notice now that any $g\in \mathcal{L}^\infty_+(\mathbb{R})$ defines a linear $(\|\cdot|_{1,+}$-$u)$-continuous functional $\varphi$ in the same manner:
$$\varphi(f)=\int_{\R} g\,f\,d\lambda\leq \int_{\R} gf^+d\lambda\leq \|g\|_\infty\int_{\R} f^+=\|g\|_\infty\|f|_{1,+},$$
which yields that $\|\varphi|^*\leq \|g\|_\infty$. On the other hand, take $\varepsilon>0$ and a set $E$ of finite measure such that $g(x)\geq \|g\|_\infty-\varepsilon$ on $E$. Then consider the function $\displaystyle{f=\frac{\mathrm{sgn}(g)}{\lambda(E)}\,\mathbf{1}_E}$, where $\mathrm{sgn}(g)$ denotes the sign of $g$, and note that $\|f|_{1,+}\leq 1$. Then
$$\varphi(f)=\frac{1}{\lambda(E)}\int_E g d\lambda\geq\frac{1}{\lambda(E)}\int_E [\|g\|_\infty-\varepsilon]d\lambda= \|g\|_\infty-\varepsilon.$$
It follows that $\|\varphi|^*=\|g\|_\infty$, and therefore, we can identify the dual of $(\mathcal{L}^1(\mathbb{R}),\|\cdot|_{1,+})$ to $(\mathcal{L}^\infty_+(\mathbb{R}), \|\cdot\|_\infty)$ by an isometric isomorphism.
\end{proof}

Let us now give the following definition.
\begin{definition}[Asymmetric weak topologies]\label{weakstar}
Let $X$ be a (asymmetric) normed space with dual $X^*$.\smallskip

\noindent (i) The \emph{weak} topology $w$ on $X$ is defined as the coarsest topology for which every $\phi\in X^*$ remains upper semicontinuous.\smallskip

\noindent (ii) The \emph{weak-star} topology $w^*$ on $X^*$ is defined as the coarsest topology that makes every evaluation functional $\{\widehat{x}:X^*\to (\mathbb{R},|\cdot|),\:x\in X\}$ continuous (notice by Remark \ref{remarkevacontin} that $\widehat{x}$ is always $\|\cdot|^*$-continuous, where $\|\cdot|^*$ is the conic hemi-norm of $X^*$).
\end{definition}
\noindent Therefore the weak-star topology $w^*$ on $X^*$  is weaker than the forward $\|\cdot|^*$-topology. In what follows, we shall make use of the notation $\langle y^*,y\rangle=y^*(y)$.

\begin{lemma}\label{farkas}
Let $X$ be an asymmetric normed space with dual $X^*$, and $\varphi:X^*\to \mathbb{R}$ a linear\linebreak $w^*$-continuous functional. Then there exists $x_{\varphi}\in X$ such that $\varphi(x^*)=x^*(x_{\varphi})$ for all $x^*\in X^*$.
\end{lemma}
\begin{proof}
Since $\varphi$ is  $w^*$-continuous, the set $\varphi^{-1}(-1,1)$ is a $w^*$-neighbourhood of 0, so there exist $x_1,...,x_n\in X$ such that
$$\{x^*_i\in X^*:\:\max_{i=1,...,n} |\langle x^*, x_i\rangle|< 1\}\subseteq \varphi^{-1}(-1,1),$$
\noindent which yields
\begin{equation}\label{ker}
\bigcap_{i=1}^n\mathrm{Ker}(\widehat{x}_i)\subseteq \mathrm{Ker}(\varphi).
\end{equation}
\noindent The above kernels are contained in the cone $X^*$. We can linearly extend $\varphi$ and the evaluation functionals $\widehat{x}_1,...,\widehat{x}_n$ from the normed cone $X^*$ to the linear space $\mathrm{span}(X^*)\subseteq \mathbb{R}^X$. This operation preserves the inclusion \eqref{ker} on the linear space $\mathrm{span}(X^*)$. It follows that the extension
$\widehat{x}_{\varphi}$ of $\varphi$ is a linear combination of the extensions of $\widehat{x}_1,...,\widehat{x}_n$.
\end{proof}
\medskip

The following result is analogous for the classical one in the operator theory (see~\cite[Theorem 4.10]{Rudin}).

\begin{lemma}\label{adjunto}
Let $(X,\|\cdot|_X)$, $(Y,\|\cdot|_Y)$ be asymmetric normed spaces, $X^*$ and $Y^*$ their respective dual cones and $T:Y^*\to X^*$ a linear bounded operator (meaning that there exists $K\geq 0$ such that $\|Ty^*|_Y\leq K\|y^*|_X$ for all $x\in X$). If $T$ is $(w^*$-$w^*)$-continuous, then there exists a linear bounded operator $S:X\to Y$ such that $T=S^*$, in the sense that
$$\langle y^*,Sx\rangle =\langle Ty^*,x\rangle,\quad\text{for all } x\in X\text{ and }y^*\in Y^*.$$
Furthermore, if $T$ is a bijective isometry, so is $S$.
\end{lemma}
\begin{proof}
Let $x\in X$, and define $f:Y^*\to \mathbb{R}$ as $f(y^*)=\widehat{x}(Ty^*)$, which is $w^*$-continuous, and therefore by Lemma \ref{farkas} there exists $y_x$ such that $\widehat{x}\circ T=\widehat{y_x}$ and $y^*(y_x)=\widehat{x}Ty^*$, and define $Sx=y_x$, which is linear and bounded, since
$$\|Sx|_Y=\|y_x|_Y=\|\widehat{y_x}|=\|\widehat{x}\circ T|=\sup_{\|y^*|\leq 1}(\widehat{x}\circ T)(y^*)\leq \|x|_X\|T|. $$
And $S^*=T$, as
$$\langle S^*y^*,x\rangle = \langle y^*,Sx\rangle=\langle \widehat{x}\circ T, y^*\rangle= \langle Ty^*,x\rangle$$
\noindent for all $x\in X$ and $y^*\in Y^*$, so $S^*=T$. Finally, if $T$ is an isometry then
$$ \|Sx|_Y=\sup_{\|y^*|\leq 1} \langle y^*,Sx\rangle=\sup_{\|y^*|\leq 1}\langle Ty^*,x\rangle=\sup_{\|y^*|\leq 1}\langle x^*,x\rangle=\sup_{\|x^*|\leq 1}\langle Ty^*,x\rangle,$$
\noindent where the first equality follows as a corollary of the Hahn-Banach theorem for asymmetric normed spaces (\cite[Corollary 2.2.4]{cobzas2012functional}).
\end{proof}

The following proposition shows that an asymmetric normed space and its bicompletion have the same dual. This fact will be relevant for our main result.
\begin{proposition}[Unique extension of a linear usc functional]\label{extension}
Let $(X,\|\cdot|)$ be an asymmetric normed space, $D\subseteq X$ a subspace that is dense in the symmetrization of the induced quasi-metric, and $\varphi:D\to \mathbb{R}$ a linear usc functional. Then $\varphi$ has a unique linear usc extension to $X$.
\end{proposition}
\begin{proof}
Thanks to the Hahn-Banach theorem \cite[Theorem 2.2.1]{cobzas2012functional}, $\varphi$ has at least one linear usc extension to $X$. Let us assume, towards a contradiction, that $\varphi$ has two different extensions $\phi_1$ and $\phi_2$, with $\phi_1(x)<\phi_2(x)$ for some $x\in X$. Since $D$ is dense for the symmetrized extended quasi-metric ({\em c.f.} Definition~\ref{sym-dis}), there is a sequence $\{x_n\}_n\subseteq D$ such that $x_n\to x$ in both $d_e$ and $\bar{d_e}$. Since $\phi_1$ and $\phi_2$ are usc for $d_e$, we deduce that they are also lsc for $\bar{d_e}$ (see Remark~\ref{rem-22}). Moreover, both functionals coincide on the sequence $\{x_n\}_n$. We deduce:
$$\limsup_n \phi_2\left(x_n\right)\leq \phi_1(x)<\phi_2(x)\leq \liminf_n \phi_2\left(x_n\right),$$
which is a contradiction. Therefore $\phi_1=\phi_2$.
\end{proof}
\medskip
\begin{proposition}[Dual of an asymmetric normed space]\label{dualbicomp}
Let $(X,\|\cdot|)$ be an asymmetric normed space and $(\tilde{X}, \|\cdot|_\sim)$ its bicompletion. Then, the respective dual cones are isometrically isomorphic.
\end{proposition}
\begin{proof}
We already know that the extension mapping from $X^\ast$ to $\tilde{X}^\ast$ is a bijection, in virtue of Proposition~\ref{extension}. To check that it is an isometry, we only need to check that $\|\phi_{|_X}|^*\geq \|\phi|^*$ for any $\phi\in\tilde{X}^\ast$, as the reverse inequality is obvious. Let $B_{\tilde{X}}$ be the unit ball of $\tilde{X}$ for the forward distance, and consider $\phi\in\tilde{X}^\ast$ and a sequence $\{z_n\}_n$ on $B_{\tilde{X}}$ such that $\phi(z_n)\to \|\phi|^*:=\sup_{z\in B_{\tilde{X}}}\phi(z)$. Since $X$ is dense for the symmetrized topology in $\tilde{X}$ (by definition), for each $n\in \N$ there exists a sequence $\{x_n^j\}_j\subseteq B_X$ such that $\{x_n^j\}$ converges to $z_n$ in the symmetrized distance of $\tilde{X}$. In particular, $\{x_n^j\}_j$  converges for both quasi-metrics $d_e$ and $\bar{d_e}$. Since $\phi$ is lsc for $\bar{d_e}$, we have that $\phi\left(z_n\right)\leq \liminf_j \phi\left(x_n^j\right)$, for every $n\in\N$. Then, for any $\varepsilon >0$ there exists $n_0\in\mathbb{N}$ such that $\|\phi|^*<\varepsilon+\phi\left(z_{n_0}\right)$, and consequently $$\|\phi|^*<\varepsilon + \liminf_j \phi\left(x_{n_0}^j\right)\leq \varepsilon + \|\phi_{|_X}|^*.$$
This completes the proof. \end{proof}


\section{The semi-Lipschitz free space}\label{main}

Throughout this section, $(X,d)$ will denote a quasi-metric space, with $d$ being possibly a quasi-hemi-metric, and with base point $x_0\in X$.

\subsection{Construction of \texorpdfstring{$\mathcal{F}_a(X)$}{Fa(X}} 
We are ready to proceed to the construction of the (asymmetric) semi-Lipschitz free space.
For every $x\in X$ we consider the corresponding {\em evaluation} mapping
$$
\delta_x:\mathrm{SLip}_0(X)\to \mathbb{R}\,\,\, \text{defined by } \delta_x(f)=f(x),\,\,\forall f\in \SLip_0(X).
$$
Notice that $\delta_x$ is a linear mapping over the cone $\SLip_0(X)$ ({\em c.f.} Definition~\ref{cone-linear}). We can also define the linear mapping $-\delta_x$ by $-\delta_x(f):=-f(x)$, for all $f\in\SLip_0(X)$.
\begin{proposition}[$\delta_x$ belongs to the linear part of $\left(\SLip_0(X)\right)^*$]\label{deltas}
For each $x\in X$, both the evaluation functional $\delta_x:\mathrm{SLip}_0(X)\to \mathbb{R}$ and its opposite $-\delta_x$ belong to the dual cone $\left(\mathrm{SLip_0}(X), \|\cdot|_S\right)^*$.
\end{proposition}
\begin{proof}
Let $x\in X$. Since $\delta_x$ is linear, we only need to check that it is bounded from above on the unit ball of $\SLip_0(X)$. Indeed, for any $f\in \mathrm{SLip}_0(X)$, we have $f(x)=f(x)-f(x_0)\,\leq \,d(x_0,x)\,\|f|_S$, therefore $\delta_x\in \mathrm{SLip_0}(X)^*$. Using the same argument, we get that $-f(x)\leq d(x,x_0)\,\|f|_S$.
\end{proof}
\begin{remark}
The fact that both $\delta_x$ and $-\delta_x$ are semi-Lipschitz yields that $\delta_x$ is actually a Lipschitz function on $\left(\SLip_0(X), \|\cdot|\right)$ of constant $\|\delta_x\|_{\mathrm{Lip}} = \max\{d(x,x_0), d(x_0,x)\}$.
\end{remark}
\medskip
\begin{proposition} [Isometric injection of $X$ into $\SLip_0(X)^*$]\label{isoinj}
The mapping $$\delta:(X,d)\to \left(\mathrm{SLip}_0(X)^*,\|\cdot|^*\right),$$ defined by $\delta(x)=\delta_x$ is (injective and) an \emph{isometry} onto its image. Therefore, for any $x,y\in X$, we have:
$$d(x,y)=\|\delta_y-\delta_x|^{*}.$$
\end{proposition}
\begin{proof}
Let us take $x,y\in X$.
First of all, it is worth noting that the quasi-metric generated by the conic-norm is extended (Proposition \ref{extendedqm}) and that $\|\delta_y-\delta_x|^{*}$ is well defined (by Proposition~\ref{deltas}). Note also that any dual cone is cancellative, since it is contained in a linear space of real-valued functions.
To prove injectivity of $\delta$, consider $x,y\in X$ such that $\delta_x=\delta_y$. Then we take the functions $f(\cdot)=d(x,\cdot)-d(x,x_0)$ and $g(\cdot)=d(y,\cdot)-d(y,x_0)$. Since $\delta_x(f)=\delta_y(f)$, and $\delta_x(g)=\delta_y(g)$, we conclude that both $d(x,y)$ and $d(y,x)$ must be zero, therefore $x=y$ (Definition~\ref{deftop}(iii)). \smallskip

\noindent By Remark~\ref{excl}(ii), for any $x,y\in X$ we have that $d_e\left(\delta_x,\delta_y\right)=\|\delta_y-\delta_x|^*$. Then, for any $x,y\in X$,
\begin{align*}
d_e\left(\delta_x,\delta_y\right)=\sup_{\|f|_S\leq1}\left(\delta_y-\delta_x\right)(f)&=\sup_{\|f|_S\leq1}\big\{f(y)-f(x)\big\}\\
&\leq \sup_{\|f|_S\leq1}\|f|_S\,d(x,y)=d(x,y).
\end{align*}
Conversely, by taking $f(\cdot)=d(x,\cdot)-d(x,x_0)$ it follows, as in Example~\ref{dist}, that
$$f(y)-f(x)=d(x,y)\ \text{and}\ f(y)-f(x)=(\delta_y-\delta_x)(f)\leq\|\delta_y-\delta_x|_*=d_e(\delta_x,\delta_y).$$
Then the result holds.
\end{proof}
\medskip
\noindent We now take the asymmetric normed space $\left(\mathrm{span}\left(\delta(X)\right),\|\cdot|^*\right)$ (which is contained in the normed cone
$\left(\mathrm{SLip}_0(X),\|\cdot|^*\right)$), and we define the (asymmetric) \emph{semi-Lipschitz free space} to be the bicompletion of $\left(\mathrm{span}\left(\delta(X)\right),\|\cdot|^*\right)$.

\begin{definition}[The semi-Lipschitz free space]\label{free}
Let $(X,d)$ be a quasi-metric space with base point~$x_0$. The \emph{semi-Lipschitz free space over }$(X,d)$, denoted by $\mathcal{F}_a(X)$, is the (unique) bicompletion of the asymmetric normed space $\left(\mathrm{span}\left(\delta(X)\right),\|\cdot|^*\right)$, where $\|\cdot|^*$ is the restriction of the norm of $\mathrm{SLip}_0(X)^*$.
\end{definition}
\medskip
\noindent We are now ready to establish our main result which is analogous of the fundamental property of the Lipschitz-free space of a metric space: being a predual of the space of Lipschitz functions vanishing at the base point.
\medskip
\begin{theorem}[$\F_a(X)^*=\SLip_0(X)$]\label{predual}
Let $(X,d)$ be a quasi-metric space with base point $x_0$. Then the dual cone of $\mathcal{F}_a(X)$ is isometrically isomorphic to $\mathrm{SLip}_0(X)$.
\end{theorem}
\begin{proof}
Thanks to Proposition~\ref{dualbicomp}, we only need to check that the dual cone of
$\left(\mathrm{span}\left(\delta(X)\right),\|\cdot^*\right)$ is isometrically isomorphic to $\mathrm{SLip}_0(X)$. To this end, we define the mapping
$$\Phi:\mathrm{SLip}_0(X)\to \left(\mathrm{span}\left(\delta(X)\right),\|\cdot|^*\right)^*,$$
with $$\Phi(f)\left(\sum_i\lambda_i\,\delta_{x_i}\right)=\sum_i\lambda_if\left(x_i\right)$$ for any linear combination of evaluation functionals. First, we check that $\Phi$ is well defined: $\Phi$ is obviously linear, so let us demonstrate condition (iii) of Proposition~\ref{caract}. For any $f\in \mathrm{SLip}_0(X)$ and any\linebreak $\sum_i\lambda_i\,\delta_{x_i}\in \mathrm{span}\left(\delta(X)\right)$, we have
$$\Phi(f)\left(\sum_i\lambda_i\delta_{x_i}\right)=\sum_i\lambda_if(x_i)=\left(\sum_i\lambda_i\delta_{x_i}\right)(f)\,\leq\, \left\|\sum_i\lambda_i\delta_{x_i}\right|^{*}\,\|f|_S.$$
Therefore $\|f|_S\geq \|\Phi(f)|^{**}$, where $\|\cdot|^{**}$ is the norm on $\left(\mathrm{span}\left(\delta(X)\right),\|\cdot|^*\right)^*$. Conversely, consider $f\in\mathrm{SLip}_0(X)$. Then, by Proposition~\ref{SLcri}, we have
\begin{align*}
\|f|_S &= \sup_{d(y,x)>0} \frac{ \max\,\{f(x)-f(y), 0\}}{d(y,x)}\\
&=\sup_{d(y,x)>0}\frac{\max\{\Phi(f)\left(\delta_x-\delta_y\right), 0\}}{\|\delta_x-\delta_y|^*}\leq \|\Phi(f)|^{**},
\end{align*}
\noindent from which we deduce that $\Phi$ is an isometry. Since $\Phi$ is obviously linear and injective, it remains only to establish surjectivity. This follows from the fact that any $\varphi \in \left(\mathrm{span}\left(\delta(X)\right),\|\cdot|^*\right)^*$ can be seen as
$\Phi\left(\varphi\circ\delta\right)$, with $\varphi\circ\delta$ being semi-Lipschitz on $X$: indeed, for every $x,y\in X$ we have:
$$\varphi\left(\delta(x)\right)-\varphi\left(\delta(y)\right)=\varphi\left(\delta_x-\delta_y\right)\leq \|\varphi|^{**}\|\delta_x-\delta_y|^*=\|\varphi|^{**}\, d(y,x).$$
This shows that $\varphi\circ\delta$ belongs to $\SLip_0(X)$ and $\Phi$ is surjective.
\end{proof}
\medskip
\begin{remark}[Compatibility with the classical theory of metric free spaces]
If $(X,d)$ is a metric space, then $\mathrm{SLip}_0(X)=\mathrm{Lip}_0(X)$. Moreover, every linear usc functional on a normed space is continuous; thus, the dual cone of a normed linear space is the same as the usual dual. We deduce that $\mathcal{F}_a(X)=\mathcal{F}(X)$.
\end{remark}
\medskip
\begin{remark}
For a quasi-metric space $(X,d)$, it is easy to check that the space of semi-Lipschitz functions for the reverse quasi-metric $\mathrm{SLip}_0(X,\bar{d})$ is exactly $-\mathrm{SLip}_0(X,d)$, and that $\|f|_S=\|\!\!-\!\!f|_{\bar{S}}$ for any ${f\in \mathrm{SLip}_0(X,d)}$, where $\|\!\!-\!\!f|_{\bar{S}}$ denotes the semi-Lipschitz constant of $-f$ on $(X,\bar{d})$. Using this isometry, we can identify the dual cones of $\mathrm{SLip}_0(X,\bar{d})$ by the isometry $\Psi$ defined by $\Psi(\mu)(f)=\mu(\!-\!f)$ for all $f\in \mathrm{SLip}_0(X,d)$, and therefore we obtain that $\mathcal{F}_a(X,d)=\Psi(\mathcal{F}_a(X,\bar{d}))$ and that $\|\Psi(\mu)|^*_{\bar{d}}=\|\!\!-\!\!\mu|^*$, where $\|\cdot|^*_{\bar{d}}$ is the norm of $\mathcal{F}_a(X,\bar{d})$.
\end{remark}\medskip



\subsection{Relation with molecules}

Given $(X,d)$ a quasi-metric space (always with a base point $x_0\in X$), we next give a description of the closed unit ball of $\mathcal{F}_a(X)$ by means of the semi-Lipschitz evaluation functionals (often called molecules)
$$M_{(x,y)}=\frac{\delta(x)-\delta(y) }{ d(y,x)},$$
where $x,y\in X$ such that $d(y,x)>0$. Let $\widehat{\mathcal{M}}_{X}$ denote the set $\{M_{(x,y)}:x,y\in X\ \mbox{with }d(y,x)>0\}$.

Before going to this, if $(X,d)$ is an asymmetric locally convex space,  it is worth noting that the \emph{asymmetric polar} of a subset $Y\subset X$ in the case of the asymmetric dual $X^{*}$ can be defined as \cite[p. 161]{cobzas2012functional}
$$Y^{\alpha}=\{\varphi\in X^{*}:\varphi(y)\leq 1,\ \mbox{for all } y\in Y\}.$$
Analogously, we can define the \emph{asymmetric polar} of a subset $W$ of the dual $X^{*}$ by \cite[p. 165]{cobzas2012functional}
$$W_{\alpha}=\{x\in X:\varphi(x)\leq 1,\ \mbox{for all } \varphi\in W\}.$$

\begin{proposition}\label{unitball}
Let $(X,d)$ be a quasi-metric space with base point $x_0$. The closed unit ball of $\mathcal{F}_a(X)$ coincides with $\left(\{M_{(x,y)}: x,y\in X:d(y,x)>0\}^{\alpha}\right)_{\alpha}.$
\end{proposition}
\begin{proof}
Let $B_{\mathrm{SLip}_0(X)}$, $B_{\F_a(X)}$ and $B_{\F_a(X)^{*}}$ denote respectively the closed unit balls of $\mathrm{SLip}_0(X)$, $\F_a(X)$ and $\F_a(X)^{*}$, and consider
the isometry
$\Phi:\mathrm{SLip}_0(X)\to \left(\mathrm{span}\left(\delta(X)\right),\|\cdot|^*\right)^*$ defined in the proof of Theorem~\ref{predual} as $$\Phi(f)\left(\sum_i\lambda_i\delta_{x_i}\right)=\sum_i\lambda_if\left(x_i\right)$$ for any linear combination of evaluation functionals. If $f\in \mathrm{SLip}_0(X)$, the condition $\|f|_S\leq 1$ is equivalent to $\frac{f(x)-f(y) }{ d(y,x)}\leq 1,\ \mbox{for all } x,y\in X$ with $d(y,x)>0$ (by Proposition~\ref{SLcri}). Since $\Phi$ is an isometry, $\|f|_S\leq 1$ also yields $\Phi(f)(M_{(x,y)})\leq 1$, for all $M_{(x,y)}\in \widehat{\mathcal{M}}_{X}$.
Hence
\begin{align*}
B_{\F_a(X)^{*}}&=
\{\Phi(f):f\in \mathrm{SLip}_0(X),\Phi(f)(M_{(x,y)})\leq 1,\ \forall M_{(x,y)}\in \widehat{\mathcal{M}}_{X}\}\\
&=\{F\in \F_a(X)^*:F(M_{(x,y)})\leq 1,\ \forall M_{(x,y)}\in \widehat{\mathcal{M}}_{X}\}=(\widehat{\mathcal{M}}_{X})^{\alpha}
\end{align*}
and thus
$$\Phi(B_{\mathrm{SLip}_0(X)})_{\alpha}=\left((\widehat{\mathcal{M}}_{X})^{\alpha}\right)_{\alpha}.$$
Moreover,
\begin{align*}
\left((\widehat{\mathcal{M}}_{X})^{\alpha}\right)_{\alpha}&=\Phi(B_{\mathrm{SLip}_0(X)})_{\alpha}\\
&=\{\gamma\in \F_a(X):\Phi(f)(\gamma)\leq 1,\ \forall f\in B_{\mathrm{SLip}_0(X)}\}\\
&=\{\gamma\in \F_a(X):\gamma(f)\leq 1,\ \forall f\in B_{\mathrm{SLip}_0(X)}\}\\
&=\{\gamma\in \F_a(X):\|\gamma|^{*}(\ =\sup_{\|f|_S\leq 1}\gamma(f)\ )\leq 1\}=B_{\F_a(X)}.
\end{align*}
The proof is complete.
\end{proof}

\begin{remark}\label{rem-AE}
Let $(X,d)$ be a quasi-metric space and $\overline{x}\notin X$. Then setting $\tilde{X}=X\cup \{\overline{x}\}$ and extending $d$ from $X\times X$ to $\tilde{X}\times \tilde{X}$ by $\tilde{d}(x,\overline{x})=\tilde{d}(\overline{x},x)=1$ and $\tilde{d}(\overline{x},\overline{x})=0$, we obtain a new quasi-metric space $(\tilde{X},\tilde{d})$ with base point $x_0\equiv\overline{x}$. Then the above construction will correspond to an asymmetric version of the Arens-Eells approach ({\em c.f.} \cite{AE1956}).
\end{remark}

\medskip

\subsection{Relation with asymmetrizations}\label{sub3.3}

Let $X=(X,D)$ be a metric space with a base point $x_{0}\in X$ and denote by
\begin{equation*}
L=(\mathrm{Lip}_{0}(X,D),\Vert \cdot \Vert _{L})  \label{eq:L}
\end{equation*}
its nonlinear dual. Let $P\subseteq L$ be a cone satisfying~\eqref{eq:P}, that is, for every $\phi\in L$ there exists $\phi _{1},\phi _{2}\in P$ such that $\phi =\phi _{1}-\phi _{2}$ and $\max \,\left\{ ||\phi _{1}||_{L},||\phi _{2}||_{L}\right\} \,\leq \,||\phi
||_L\,\leq \,||\phi _{1}||_{L}+||\phi _{2}||_{L}.$ Let us denote by $D_{P}$ the $P$-asymmetrization of $X$ (\textit{c.f.} Definition~\ref{canoasym}). We also denote by
\begin{equation*}
SL=(\mathrm{SLip}_{0}(X,D_{P}),\Vert \cdot |_{S})  \label{eq:SL}
\end{equation*}
the nonlinear asymmetric dual of $(X,D_P)$, that is, the normed cone of semi-Lipschitz
functions on~$(X,D_{P})$.

\begin{lemma}[Isometric injection of $P$ into $SL$]
\label{isoinjec} For every metric space $(X,D)$ and every $P$-asymmetrization $(X,D_{P})$:\smallskip \newline
{\rm (i)} there exists an isometric injection of $P$ into $SL$; \smallskip
\newline
{\rm (ii)} there is a non-expansive injection of $SL$ into $L$.
\end{lemma}

\begin{proof}
Let $\phi \in SL$ and $x,y\in X$. Then
\begin{equation*}
\phi (y)-\phi (x)\leq \Vert \phi |_{S}D_{P}(x,y)=\Vert \phi |_{S}\Vert \delta_y-\delta_x|_{\mathcal{F}_{P}}\leq \Vert \phi |_{S}\Vert \delta_y-\delta_x\Vert_{\mathcal{F}}=\Vert \phi |_{S}D(x,y),
\end{equation*}
which yields that $\phi \in \mathrm{Lip}_{0}(X,D)$ and $\Vert \phi \Vert_{L}\leq \Vert \phi |_{S}$. This proves (ii).\smallskip\newline
Let now $\phi :(X,D)\rightarrow \mathbb{R}$ be in $P$ with $\Vert \phi\Vert_{L} \neq 0$. Then
$\phi_{1}=\frac{\phi }{\Vert \phi \Vert_L}$ is also in $P$ and $\Vert \phi_{1}\Vert_{L}=1$. Given $x,y\in X$, we deduce:
\begin{align*}
D_{P}(x,y)=\Vert \delta_y-\delta_x|_{\mathcal{F}_{P}}&=\sup_{\substack{ \psi
\in P \\ \Vert \psi \Vert _{L}\leq 1}}(\psi (y)-\psi (x))\\
&\geq\phi_{1}(y)-\phi _{1}(x)
=\frac{1}{\Vert \phi \Vert _{L}}(\phi (y)-\phi (x)),
\end{align*}
which yields $\phi (y)-\phi (x)\leq \Vert \phi \Vert _{L}D_{P}(x,y)$.
Therefore, $\phi \in SL$ and $\Vert \phi |_{S}\leq \Vert \phi \Vert _{L}$.
Combining with (ii) we deduce  $\Vert \phi \Vert _{L}=\Vert \phi |_{S}$
and (i) follows.
\end{proof}

Let us set
\begin{equation}
F=\mathrm{span}\{\delta (x):x\in X\}\subset SL^{\ast }\text{ \quad and\quad }
\widehat{F}=\mathrm{span}\{\widehat{\delta}(x):x\in X\}\subset L^{\ast }
\label{eq:FF}
\end{equation}
where $\delta $ (respectively, $\widehat{\delta}$) is the canonical injection of
$(X,D_{P})$ into $SL^{\ast }$ (respectively, of $(X,D)$ into $L^\ast $).
There is a canonical bijection between $F$ and $\widehat{F},$ under which a
general element $Q=\sum_{i=1}^{n}\lambda _{i}\delta (x_{i})$ of $F$ is
identified with the element $\widehat{Q}=\sum_{i=1}^{n}\lambda _{i}\widehat{\delta}
(x_{i})$ of $\widehat{F}.$ Using this bijection, we have the following result.

\begin{proposition}[ $||\cdot ||_{\mathcal{F}}$ is equivalent to the
symmetrization of $\Vert \cdot |_{\mathcal{F}_{a}}$]
\label{Prop_equiva}For any $Q\in F$ it holds:
\begin{equation*}
\max \{\Vert Q|_{\mathcal{F}_{a}},\Vert\!\!-\!\!Q|_{\mathcal{F}_{a}}\}\leq \Vert
\widehat{Q}\Vert _{\mathcal{F}}\leq \Vert \widehat{Q}|_{\mathcal{F}_{P}}+\Vert -\widehat{
Q}|_{\mathcal{F}_{P}}\leq 2\max \{\Vert Q|_{\mathcal{F}_{a}},\Vert\!\!-\!\!Q|_{
\mathcal{F}_{a}}\}.
\end{equation*}
\end{proposition}

\begin{proof}
Let $F=\mathrm{span}\left( \delta (X)\right) $ and $Q\in F$. Then $Q$ is of
the form $Q=\sum_{i=1}^{n}\lambda _{i}\delta (x_{i})$ for some $n\in \mathbb{
N}$, $\lambda _{i}\in \mathbb{R}$ and $x_{i}\in X$, $i=1,\ldots ,n$, and
\begin{equation*}
\Vert \widehat{Q}|_{\mathcal{F}_{P}}=\sup_{\substack{ \phi \in P \\ \Vert \phi
\Vert _{L}\leq 1}}\langle \phi ,\widehat{Q}\rangle =\sup_{\substack{ \phi \in P
\\ \Vert \phi \Vert _{L}\leq 1}}\sum_{i=1}^{n}\lambda _{i}\phi (x_{i})\leq
\sup_{\substack{ \varphi \in SL \\ \Vert \phi |_{S}\leq 1}}%
\sum_{i=1}^{n}\lambda _{i}\varphi (x_{i}):=\Vert Q|_{\mathcal{F}_{a}}.
\end{equation*}%
We also obtain $\Vert -\widehat{Q}|_{\mathcal{F}_{P}}\leq \Vert\!\!-\!\!Q|_{\mathcal{F}_{a}}$. 
Now, if $\varphi \in SL$ satisfies $\Vert \phi |_{S}\leq 1,$ then by
Lemma~\ref{isoinjec}(ii) we deduce that $\varphi \in L$ and $\Vert
\varphi \Vert _{L}\leq \Vert \varphi |_{S}\leq 1$. Hence
\begin{equation*}
\Vert Q|_{\mathcal{F}_{a}}\leq \sup_{\substack{ \phi \in L \\ \Vert \phi
\Vert _{L}\leq 1}}\sum_{i=1}^{n}\lambda _{i}\phi (x_{i})=\Vert \widehat{Q}\Vert
_{\mathcal{F}}
\end{equation*}
and $\Vert\!\!-\!\!Q|_{\mathcal{F}_{a}}\leq \Vert \widehat{Q}\Vert _{\mathcal{F}}$,
which yields
\begin{equation*}
\max \{\Vert Q|_{\mathcal{F}_{a}},\Vert\!\!-\!\!Q|_{\mathcal{F}_{a}}\}
\leq\Vert \widehat{Q}\Vert _{\mathcal{F}}\leq \Vert \widehat{Q}|_{\mathcal{F}_{P}}+\Vert -\widehat{Q}|_{\mathcal{F}_{P}},
\end{equation*}
where the last inequality follows from \eqref{eq:equiv}. The result follows.
\end{proof}

\medskip

In the sequel, we shall identify $F$ with $\widehat{F}$, defined in \eqref{eq:FF}. Under this identification, the norm $\|\cdot \|_{\mathcal{F}}$ can be considered to be also defined on $F$. Under this convention, the
statement of Proposition~\ref{Prop_equiva} reads as follows: the norm $||\cdot ||_{\mathcal{F}}$ is equivalent to the symmetrization of $\Vert
\cdot |_{\mathcal{F}_{a}}\mathcal{\ }$and consequently
\begin{equation}\label{eq:F}
\mathcal{F}_{a}(X,D_{P})=\overline{F}^{\Vert \cdot \Vert _{\mathcal{F}
_{a}^{s}}}=\overline{F}^{\Vert \cdot \Vert _{\mathcal{F}}}=\mathcal{F}(X,D),
\end{equation}
which yields that $\mathcal{F}_{a}(X,D_{P})$ and $\mathcal{F}(X,D)$ can be
identified as sets. Moreover
\begin{equation*}
D_{P}(x,y)=\Vert \delta (y)-\delta (x)|_{\mathcal{F}_{P}}=\Vert \delta
(y)-\delta (x)|_{\mathcal{F}_{a}}.
\end{equation*}
Hence the following result holds.

\begin{theorem}[Compatibility I]
\label{Nino} Let $(X,D)$ be a metric space with a $P$-asymmetrization. Then, the symmetrizations of $(\mathcal{F}(X,D),\|\cdot|_{\mathcal{F}_P})$
and $(\mathcal{F}_a(X),\|\cdot|_{\mathcal{F}_a})$ are both isomorphic to the
Lipschitz-free space $(\mathcal{F}(X,D),\|\cdot\|_\mathcal{F})$.
\end{theorem}

The following diagram illustrates the situation described by Theorem~\ref{Nino}.\medskip

$$
\begin{tikzpicture}
  \node (X) {$\ \ F=\mathrm{span}\left(\delta(X)\right)$};
  \node (FX1) [below of=X] {};
 \node (Y) [right of=X] {};
 \node (Z) [right of=Y] {$\sqsubseteq$};
 \node (N) [below of=Z] {\footnotesize{$\|\cdot\|_{\mathcal{F}_a^s}$-dense}};
   \draw[<-] (Z) to node [swap] {$$} (N);
  \node (Z2) [right of=Z] { $\mathcal{F}_a(X)$ };
 \node (Z3) [right of=Z2] {$\sqsubseteq$};
 \node (Z4) [right of=Z3] {$(SL)^*$};
  \node (FX) [below of=FX1] {$\widehat{F}=\mathrm{span}\left(\widehat{\delta}(X)\right)$};
  \node (E1) [right of=FX] { };
   \node (E3) [right of=E1] {$\sqsubseteq$};
     \node (E4) [right of=E3] { $\mathcal{F}(X)$ };
   \node (E5) [right of=E4] {$\sqsubseteq$};
     \node (E6) [right of=E5] { $L^*$ };
  \draw[<->] (X) to node [swap] {$$} (FX);
 \node (N3) [below of=E3] {\footnotesize{$\|\cdot\|_{\mathcal{F}}$-dense}};
    \draw[->] (N3) to node [swap] {$$} (E3);
\end{tikzpicture}
$$

\begin{center}
\noindent \rule{10cm}{1.2 pt}
\end{center}

\medskip

\noindent Let us now study the inverse procedure: we start with a quasi-metric space $(X,d)$ and consider a symmetrization $D$ of its distance (where $D$ is either $d^s$ or $d^{s_0}$, see Remark~\ref{sym-dis}). It is easily seen that every $\phi\in\SLip_0(X,d)$
satisfies $\phi\in\Lip_0(X,D)$ and $\|\phi\|_L\leq \|\phi|_S$. Therefore, $P:=\SLip_0(X,d)$ can be viewed as a cone in $\Lip_0(X,D)$ and be used to define an asymmetric norm $\|\!\cdot\!|_{\F_P}$ on $\mathcal{F}(X,D)$ and consequently a quasi-metric $D_P$ on $X$. In this setting, forthcoming Proposition~\ref{compatibilyII} establishes a compatibility result under the following assumption: \smallskip\newline

\noindent There exists $\alpha \geq 1$ such that for every $\phi\in \SLip_0(X,d)$
 \begin{equation}\label{eq-cII}
 \Big(\,\|\phi\|_L\leq\Big)\,\, \|\phi|_S\leq \alpha \|\phi\|_L\,.
\end{equation}

\begin{proposition}[Compatibility II]\label{compatibilyII}
Let $(X,d)$ be a quasi-metric space with symmetrized distance $D$ and assume~\eqref{eq-cII} holds.
We set $P:=\SLip_0(X,d)$ and define for every $Q\in\F(X,D)$
$$
\|Q|_{\F_P}:=\displaystyle\sup_{\substack{\phi \in P\\ \|\phi\|_L \leq 1}}\langle Q,\phi \rangle\,.
$$
Then for all $Q\in \mathrm{span}(\delta(X))$
\begin{equation}\label{eq-eq-an}
\|Q|_{\F_a}\leq \|Q|_{\F_P}\leq \alpha \|Q|_{\F_a}\,.
\end{equation}
In particular, setting for $x,y\in X$
$$D_P(x,y):=\|\delta_y-\delta_x|_{\F_P}$$
we obtain that for all $x,y\in X$ it holds:
\begin{equation}\label{eq-eq-qd}
 d(x,y)\leq D_P(x,y)\leq \alpha d(x,y)\,.
\end{equation}
\end{proposition}
\noindent {\bf Terminology} ({\em equivalence of asymmetric norms/quasi-metrics}). We interpret relation~\eqref{eq-eq-an} as an equivalence relation for the asymmetric norms $\|Q|_{\F_a}$ and $\|Q|_{\F_P}$. Similarly, relation~\eqref{eq-eq-qd} means that the quasi-distances $d$ and $D_P$ are equivalent.
\begin{proof}
The equivalence between the asymmetric norms $ \|\!\cdot\!|_{\F_a}$ and $\|\!\cdot\!|_{\F_P}$ on the vector space $\mathrm{span}(\delta(X))$ follows directly from their definitions and the inequalities $\|\phi\|_L\leq\|\phi|_S\leq \alpha \|\phi\|_L$.
\end{proof}

\begin{remark}
The equivalence between the quasi-metric $d$ and the canonical asymmetrization~$D_P$ of the symmetrized distance $D$ yields an equivalence between $D$ and the symmetrization $(D_P)^s$ of~$D_P$.
\end{remark}
\smallskip

\noindent If in addition to~\eqref{eq-cII}, we assume that $P=\SLip_0(X,d)$ induces an asymmetrization on the free space $\F(X,D)$, that is, for every $\phi \in \Lip_0(X,D)$ there exist
$\phi _{1},\phi _{2}\in P$ such that
\begin{equation*}
 \phi =\phi _{1}-\phi _{2}\qquad\text{ and } \qquad \max \,\left\{ ||\phi _{1}||_{L},||\phi _{2}||_{L}\right\} \,\leq \,||\phi
||_L\,\leq \,||\phi _{1}||_{L}+||\phi _{2}||_{L}
 \label{decomp}
\end{equation*}
then the equivalence between $D$ and $(D_P)^s$ extends to the corresponding free spaces (see Remark~\ref{naturalasymmetrizednorm}). In particular, the following result holds.

\begin{proposition}[Compatibility III]\label{compatibilityIII}
Let $(X,d)$ be a quasi-metric space and $D=d^s$ or $D=d^{s_0}$. Assume that the cone $P=\SLip_0(X,d)$ of $\Lip_0(X,D)$ induces an asymmetrization in $\F(X,D)$ and \eqref{eq-cII} holds. Then the asymmetric free spaces $\F_a(X,d)$ and $\F_a(X,D_P)$ coincide (as sets) with the free space $\F(X,D)$, that is:
$$\F_a(X,d)=\F(X,D)=\F_a(X,D_P).$$
Moreover: \smallskip\newline
{\rm (i).} The quasi-metrics $d$ and $D_P$ are equivalent, as also are the (symmetric) metrics $D$, $(D_P)^s$ and
$(D_P)^{s_0}$ (symmetrizations of $D_P$).\smallskip\newline
{\rm (ii).} The asymmetric norms $\|\cdot|_{\F_a(X,d)}$, $\|\cdot|_{\F_P}$ and $\|\cdot|_{\F_a(X,D_P)}$ are equivalent. \smallskip\newline
{\rm (iii).} The symmetrizations of $\|\!\cdot\!|_{\F_a(X,d)}$, $\|\!\cdot\!|_{\F_P}$ and $\|\!\cdot\!|_{\F_a(X,D_P)}$ are equivalent to
$\|\!\cdot\!\|_\F$.

\end{proposition}

\begin{proof} We have already seen that $\F(X,D)=\F_a(X,D_P)$ (as sets), see~\eqref{eq:F}. By Proposition~\ref{compatibilyII},
the asymmetric norms $\|\!\cdot\!|_{\F_a(X,d)}$ and $\|\!\cdot\!|_{\F_P}$ are equivalent on $F=\mathrm{span}\{\delta (x):x\in X\}$, therefore $$\F_a(X,d)=\overline{F}^{\|\cdot\|_{\F_a(X,d)}^s}=\overline{F}^{\|\cdot\|_{\F_P}^s}=\F(X,D).$$
Assertions (i) follow directly from Proposition~\ref{compatibilyII}. For (ii) it remains to prove that $\|\!\cdot\!|_{\F_a(X,d)}$ and $\|\!\cdot|_{\F_a(X,D_P)}$ are equivalent. We established in~\eqref{eq-eq-qd} that the quasi-distances $d$ and $D_P$ are equivalent. This yields that the normed cones $\SLip_0(X,d)$ and $\SLip_0(X,D_P)$ are isomorphic, which leads to an isomorphism of the corresponding semi-Lipschitz free spaces. The equivalence between the symmetrizations of the asymmetric norms asserted in (iii) now follows from~(ii). Thanks to Theorem~\ref{Nino} they are also equivalent to $\|\!\cdot\!\|_\F$.
\end{proof}

\begin{remark}
If the value of $\alpha$ associated to the assumption \eqref{eq-cII} is equal to $1$, all of the aforementioned equivalences of Proposition~\ref{compatibilityIII} become equalities.
\end{remark}

\medskip

\subsection{Properties \texorpdfstring{$\mathbf{(S)}$}{(S)} and \texorpdfstring{$\mathbf{(S^*)}$}{(S*)}}

We have shown that the $P$-asymmetrization of a metric space $(X,D)$ gives rise to a quasi-metric space, for which the symmetrization of its asymmetric free space is isomorphic to the free space
$(\mathcal{F}(X),\|\!\cdot\!\|_{\mathcal{F}})$.
In this subsection we shall be interested in cases in which the aforementioned isomorphism is in fact an isometry.

\begin{definition}
\label{propSS*} Let $(X,D)$ be a metric space, $L=\Lip_0(X,D)$ and $P\subset L$ be a cone.
 We say that the metric space $(X,D)$ satisfies:\smallskip
\begin{itemize}
\item[(i)]  property $\mathbf{(S)}$ with respect to $P$, if $P$ induces a nontrivial asymmetrization $D_P$ on $X$ and $$SL=\mathrm{SLip}_0(X,D_P)=P.$$
\item[(ii)] property $\mathbf{(S^\ast)}$ (respectively, $\mathbf{(S_0^\ast)}$) with respect to $P$ if, in addition to (i), it holds:
$$\|Q\|_\mathcal{F}=\|Q|_{\mathcal{F}_P}+\|\!\!-\!\!Q|_{\mathcal{F}_P} \ \left(\text{respectively,}\,\,\, \|Q\|_\mathcal{F}=\max\{\|Q|_{\mathcal{F}_P},\|\!\!-\!\!Q|_{\mathcal{F}_P}\!\}\right)$$
for every $Q\in \mathcal{F}(X,D)$.
\end{itemize}
\end{definition}
\smallskip

\noindent The following proposition is straightforward.

\begin{proposition}
\label{SS*} Let $(X,D)$ be a metric space. \smallskip \newline
\noindent {\rm (i).} If $(X,D)$ satisfies $\mathbf{(S)}$ with respect to $P$, then $(\mathcal{F}(X,D),\|\cdot|_{\mathcal{F}_P})$ and $(\mathcal{F}_a(X,D_P),\|\cdot|_{\mathcal{F}_a})$ are identical. \smallskip \\[0.05cm]
\noindent {\rm (ii).} If $(X,D)$ satisfies $\mathbf{(S^*)}$ $($resp. $\mathbf{(S_0^\ast)})$ with respect to $P$, then
the $d^s$-symmetrization (resp. \linebreak $d^{s_0}$-symmetrization) of $(\mathcal{F}_a(X,D_P),\|\cdot|_{\mathcal{F}_a})$ given in \eqref{eq-dist} is isometrically isomorphic to $(\mathcal{F}(X,D),\|\!\cdot\!\|_\mathcal{F})$.
\end{proposition}
\smallskip

\noindent Before we proceed, let us give examples of metric spaces for which the above properties fail.
\begin{example}\label{example_S}
{\bf (i)} (Property $\mathbf{(S)}$ fails) Let $X=\mathbb{R}$ with the usual distance $D(t,s)=|s-t|$, for $t,s\in\mathbb{R}$.
Let $L$ be the space of Lipschitz functions on $\mathbb{R}$ vanishing at $0$ and set
$$P:=\{\phi\in L:\; \int_{\mathbb{R}} \phi \in [0,+\infty]\}.$$
Then $P$ contains the cone $L_+$ of non-negative Lipschitz functions vanishing at $0$, and consequently $L=P-P$ and~\eqref{eq:P} holds. It is easy to see that $$D_P(t,s)=\sup_{\substack{ \phi \in P \\ \Vert \phi \Vert _{L}\leq 1}} (\phi(s)-\phi(t)) = |s-t|=D(t,s).$$ Therefore, $SL=L \neq P$ and  $\mathbf{(S)}$ fails.  \newline

\noindent {\bf (ii)} (Property $\mathbf{(S)}$ holds but properties $\mathbf{(S^\ast)}$ and $\mathbf{(S_0^{\ast})}$ fail) We consider again $X=\mathbb{R}$ equipped with its usual distance $D$ and $L$ be the space of Lipschitz functions on $\mathbb{R}$ vanishing at $0$. We now set $$P=L_+:=\{\phi\in L:\; \phi \geq 0\}.$$
It follows easily that
\begin{align*}
D_+(s,t)&=\sup_{\substack{ \phi \in L_+ \\ \Vert \phi \Vert _{L}\leq 1}} (\phi(t)-\phi(s))\,\\
&=
\begin{cases}
\,\quad |t-s|\,,\smallskip&\ \mbox{if }\ 0\leq s \leq t\,\,\,\mbox{or}\,\,\, s\leq t\leq 0\\
\,\, \min\{\,t, s-t \,\}\,,\smallskip&\ \mbox{if }\ 0\leq t\leq s\\
\,\, \min\{|s|, s-t \}\,,\smallskip&\ \mbox{if }\ t\leq s\leq 0\\
\qquad |t|\,,\smallskip&\ \mbox{if }\ t \leq 0 \leq s  \,\,\,\mbox{or}\,\,\, s\leq 0\leq t\,\\
\end{cases}.
\end{align*}
Let us show that property $\mathbf{(S)}$ holds. Indeed, for $s\neq 0$ we have $D_+(0,s)=s$ and $D_+(s,0) =0$. By Lemma~\ref{isoinjec}(i), $P\subset SL \subset L$. Let $\varphi: \mathbb{R}\mapsto\mathbb{R}$ be any function vanishing at $0$ and assume that for some $s\neq 0$ we have $\varphi(s) <0$. Then $\varphi(0)-\varphi(s) >0$ and $D_+(s,0)=0$ reveals that $\varphi$ cannot be in $SL$, showing that $\mathbf{(S)}$ holds. \smallskip\newline
Taking now any two integers $n,k\geq 2$ we have $D_+(n,-k)=k$, $D_+(-k,n)=n$ and
$D(n,-k)=n+k$, which shows that $\mathbf{(S_0^{\ast})}$ fails. On the other hand, $D_+(1,n)=n-1=D(1,n)$ and $D_+(n,1)=1$
which shows that $\mathbf{(S^\ast)}$ fails.
\end{example}
\medskip

\noindent A typical example of a metric space for which $(S^{\ast })$ holds is the set of real numbers $\mathbb{R}$ viewed as a pointed metric space, for the cone $P=\{\phi \in L:\phi ^{\prime }\geq 0\}$, see forthcoming Lemma~\ref{slip0}. To obtain additional examples of metric spaces satisfying $\mathbf{(S^\ast)}$, let us first recall definitions and results due to Godard~\cite{Godard}, regarding $\R$-trees.
\medskip
\begin{definition}[$\R$-tree]\label{deftrees}
An $\R$-tree is a metric space $T$ satisfying the following two conditions:
\begin{enumerate}
\item [(i)] For any points $x,y\in T$, there exists a unique isometry $\phi:=\phi_{xy}$ of the closed interval $[\, 0 \, , d(x,y) \, ]$ into $T$ such that $\phi(0)=x$ and $\phi(d(x,y))=y$. \\(The range of this isometry is called segment and is denoted by $[x,y]$.)
\item[(ii)]  Any one-to-one continuous mapping $\varphi : [\, 0 \, , 1 \, ] \to T$ has same range as the isometry $\phi_{a,b}$ associated to the points $a=\varphi(0)$ and $b=\varphi(1)$.
\end{enumerate}
\end{definition}

\noindent Our aim is to prove that subsets of (pointed) $\R$-trees satisfy property $\mathbf{(S^*)}$. The base point of an $\R$-tree is denoted by $0$. Then one defines a partial order $\preccurlyeq$ on $T$, by setting $x\preccurlyeq y$ if $x\in [0,y]$. \smallskip \newline
A subset $A$ of $T$ is said to be {\em measurable} whenever $\phi_{xy}^{-1}(A)$ is Lebesgue-measurable for any $x$ and $y$ in $T$. If $A$ is measurable and $S$ is the segment $[ \, x \, , y \, ]$, we write $\lambda_S(A)$ for $\lambda \left(\phi_{xy}^{-1}(A) \right)$, where $\lambda$ is the Lebesgue measure on $\R$. Let $\mathcal{R}$ be the family of all subsets of $T$ which can be written as a finite union of disjoints segments, and for $R=\bigcup_{k=1}^{n} S_k \in \mathcal{R}$, we set
$$\lambda_R(A)=\sum_{k=1}^n \lambda_{S_k}(A).
$$

\noindent Then,
$$\lambda_T(A)=\sup_{R \in \mathcal{R}} \lambda_R(A),$$
\noindent defines a measure (called the length measure) on the $\sigma$-algebra of $T$-measurable sets such that
\begin{equation*}
\int_{[ x , y ]} f(u)  d\lambda_T(u)=\int_{0}^{d(x , y)} f(\phi_{xy}(t)) dt
\end{equation*}
for any $f \in L_1(T)$ and $x,y\in T$.

\begin{definition}[measure on an $\mathbb{R}$-tree]\label{mu}
Let $T$ be a pointed $\R$-tree, and let $A$ be a closed subset of $T$. We denote by $\mu_A$ the measure defined by
$$
\mu_A=\lambda_A + \sum_{a \in A} L(a)  \delta_{a},
$$
where $\lambda_A$ is the restriction of the length measure $\lambda_T$ to $A$, $\delta_{a}$ is the Dirac measure on $a$ and\linebreak $L(a)=\inf_{x \in A \cap [ 0 , a )} d(a \, , x)$.
\end{definition}
\begin{proposition}\cite[Proposition~2.3]{Godard}\label{dualL1}
	Let $T$ be a pointed $\R$-tree, and let $A$ be a closed subset of $T$ containing $0$. Then, $\mathcal{L}^1(\mu_A)^\ast$ is isometrically isomorphic to $\mathcal{L}^\infty (\mu_A)$.
\end{proposition}

\begin{definition}[Differentiation on an $\mathbb{R}$-tree] \label{diff}
Let $T$ be a pointed $\R$-tree, $A$ a closed subset of $T$ containing $0$ and $f : A \to \R$. For $a \in A$, let $\tilde{a}$ be the unique point in $[0 , a ]$ such that $d(a,\tilde{a})=L(a)$. If $L(a)>0$, we say that $f$ is differentiable at $a$ with derivative $$f'(a)= \dfrac{f(a)-f(\tilde{a})}{L(a)}.$$ If $L(a)=0$, we say that $f$ is differentiable at $a$, whenever the limit
$$
\lim_{\substack{x \to \tilde{a} \\ x \in [0 , a) \cap A}} \dfrac{f(a)-f(x)}{d(x,a)}
$$
exists, and we denote by $f'(a)$ the value of this limit.
\end{definition}
\begin{proposition}\cite[pp.~4313-4314]{Godard}
	Let $f$ be a real valued Lipschitz function defined on an $\R$-tree $T$. Then, $f$ is differentiable almost everywhere on $T$ and

$$
f(x)-f(0)=\int_{[ 0 , x ]} f' d \lambda_T,
$$
\noindent for all $x\in T$.
\end{proposition}

\noindent The following theorem characterizes subsets of $\R$-trees in terms of their Lipschitz-free spaces.

\begin{theorem}\cite[Theorem~4.2]{Godard}\label{godardtrees}
Let $(X,D)$ be a complete pointed metric space. Then the following assertions are equivalent:
\begin{itemize}
\item[(i).] $\mathcal{F}(X)$ is isometrically isomorphic to a subspace of an $\mathcal{L}^1$-space;
\item[(ii).] $(X,D)$ isometrically embeds into an $\R$-tree.
\end{itemize}
\end{theorem}
We are now ready to prove our result on $\R$-trees.

\begin{proposition}
\label{L1(T)} Let $(X,D)$ be a subset of an $\R$-tree $T$. Then, $(X,D)$
satisfies property $(S^{\ast })$ with respect to the cone
\begin{equation*}
P=\{\phi \in \mathrm{Lip}_0(X,D)\,:\,\phi' \geq 0\},
\end{equation*}
\end{proposition}

\begin{proof}
Thanks to Theorem \ref{godardtrees}, we may use Godard's embedding, denoted by $\Phi_\ast$, to isometrically identify $\mathcal{F}(X,D)$ with a subspace $Y$ of $\mathcal{L}^1(T)$, by sending $\delta_x\in \mathcal{F}(X,D)$ to $\Phi_\ast(\delta_x)=\mathds{1}_{[0,x]}\in \mathcal{L}^1(T)$. This embedding is the restriction to $\mathcal{F}(X,D)$ of the pre-adjoint of the (weak-star to weak-star continuous) isometry $\Phi:\mathcal{L}^\infty(T)\to \mathrm{Lip}_0(T)$ defined by $\Phi(g)(x)=\int_{[0,x]} g d\mu_X$ for any $x\in T$.

Let $\iota:(X,D)\to (Y, \|\cdot\|_1)$ be the isometric injection induced by Godard's embedding $\Phi$. We keep the same notation
$\Vert \cdot |_{\mathcal{F}_{P}}$ for the asymmetric hemi-norm induced in $Y$ by this embedding. The $P$-asymmetrization of the norm of $Y$ is given by
\begin{equation*}
\Vert f|_{\mathcal{F}_{P}}=\sup_{\substack{ \phi \geq 0 \\ \Vert \phi \Vert
_{\infty }\leq 1}}\langle \phi ,f\rangle =\sup_{\substack{ \phi \geq 0 \\
\Vert \phi \Vert _{\infty }\leq 1}}\int_{X }f\phi  d\mu_X=\int_{X } f^+:=\Vert f|_{1,+},
\end{equation*}
\noindent for all $f\in Y$, where $f^+(t)=\max\{f(t),0\}$ for any $t\in T$. Therefore, $D_{P}(y,x)=\Vert \iota (x)-\iota
(y)|_{\mathcal{F}_{P}}=0$ whenever $\iota (x)\leq \iota (y)$ almost everywhere, which is equivalent to $x\preccurlyeq y$ in the order of~ $T$. Then, for $\varphi \in SL=\mathrm{SLip}_{0}(X,D_{P})$ and $x,y\in X$ such that $x\preccurlyeq y$, we have $\varphi(x)-\varphi
(y)\leq \|\varphi|_S D_P(y,x)=0$, and therefore $x\preccurlyeq y$ yields $\varphi(x)\leq \varphi(y)$.\smallskip\newline
It is easy to check that $\Phi^{-1}(\varphi)=\varphi'\in \mathcal{L}_\infty(T)$ for all $\varphi\in L$. The monotonicity property of semi-Lipschitz functions proved above yields that $\varphi'\geq 0$, so $\varphi$ belongs to the cone $P$. Therefore, $SL\subset P$ and in view of Lemma~\ref{isoinjec}(i) we deduce that $SL=P$, hence $(X,D)$ satisfies property $\mathbf{(S)}$.\smallskip\newline
Let $g\in \mathcal{F}(X,D)$, and $f=\Phi_\ast(g)$. Then
\begin{align*}
\Vert g\Vert _{\mathcal{F}}= \Vert f \Vert_1 &=\sup_{\Vert \phi \Vert _{\infty}\leq 1}\langle
\phi ,f\rangle =\langle f,\mathrm{sgn}(f)\rangle \\
&=\langle f^{+},\mathrm{sgn}
(f)\rangle -\langle f^{-},\mathrm{sgn}(f)\rangle =\Vert f|_{1,+}+\Vert
-f|_{1,+}\\
&=\Vert g|_{\mathcal{F}_P}+\Vert
-g|_{\mathcal{F}_P},
\end{align*}
\noindent where $\mathrm{sgn}(f)$ denotes the sign of $f$. We conclude that $(X,D)$ satisfies property $\mathbf{(S^\ast)}$.
\end{proof}
Combining Propositions~\ref{SS*} and \ref{L1(T)}, we obtain

\begin{proposition}
\label{Rtrees} Let $(X,D)$ be a subset of an $\mathbb{R}$-tree. Then, there exists a canonical
asymmetrization $D_P$ of $D$ such that the symmetrization of the semi-Lipschitz free space $
\mathcal{F}_a(X,D_P)$ is isometrically isomorphic to $\mathcal{F}(X,D)$.
\end{proposition}

\section{Linearization of semi-Lipschitz functions: a universal property}\label{linearization}

As it was pointed out in Remark~\ref{remarkim}, Definition~\ref{slip} (as well as Definitions~\ref{def-SL-norm} and~\ref{Aps}) can be readily generalized to functions between quasi-metric spaces, as well as to the case of semi-Lipschitz functions with values in a normed cone.

Let $(C,\|\cdot|)$ be a normed cone, and denote as $d_e^c(u,v)$ its corresponding extended quasi-metric ({\em c.f.} Proposition \ref{extendedqm}). We next introduce the notion of semi-Lipschitz function with values in $C$.

\begin{definition} (Semi-Lipschitz function with values in a normed cone).
Let $(X,d)$  be a quasi-metric space. A function
$f:X\to C$ is said to be a semi-Lipschitz function if there exists $L\geq 0$ such that for every $x,y\in X$ we have
\begin{equation}\label{SL2}
d_e^c(f(y),f(x))\leq L d(y,x). 
\end{equation}
In this case, the semi-Lipschitz conic-norm of a function $f:X \to C$ is defined by
$$\|f|_S:=\inf\big\{L>0: \,\,\eqref{SL2} \text{ holds}\big\}.$$
The class of semi-Lipschitz functions on $X$ with values in $C$ is denoted as $\SLip(X,C)$.
Also, if $x_0\in X$ is a base point, we define the asymmetric pivot space
$$\mathrm{SLip}_0(X,C):=\{f\in \SLip(X,C)\text{ such that }f(x_0)=0\}.$$
\end{definition}\smallskip

As in Proposition \ref{SLcri}, a function $f:X\to C$ is semi-Lipschitz if and only if $\|f|_S<\infty$. Moreover, if $d$ is a quasi-metric and $f:X\to C$ is semi-Lipschitz, then
$$\|f|_S=\sup_{x\neq y}\frac{\max\{d_e^c(f(y),f(x)),0\}}{d(y,x)}=\sup_{x\neq y}\frac{d_e^c(f(y),f(x))}{d(y,x)}<\infty.$$
If $d$ is a quasi-hemi-metric, we
use the same expression for $\|f|_S$ as that in Proposition \ref{SLcri}(ii).
\smallskip

Given a quasi-metric space $(X,d)$ with base point $x_0$, for the following result consider the isometric injection $\delta:(X,d)\to\left(\mathrm{SLip}_0(X)^*,\|\cdot|^*\right)$ of Proposition \ref{isoinj}.
We next show that the semi-Lipschitz free space over a quasi-metric space $(X,d)$ with base point $x_0$ is characterized by the following universal property, which is an analog of the Lipschitz case (see \cite[Lemma~2.2]{GK_2003}).

\begin{theorem} [Linearization of semi-Lipschitz functions]\label{Lslf}
Let $(X,d)$ be a quasi-metric space with base point $x_0$. Suppose that $(C,\|\cdot|)$ is a normed cone and $f\in \mathrm{SLip}_0(X,C)$. Then there exists a unique linear map $T_f:\F_a(X)\to C$ extending $f$, i.e. $T_f \circ \delta=f$ and  $\|T_f|= \|f|_S$.
\end{theorem}
\begin{proof}
If $f\in\SLip_0(X,C)$, then $T_f\colon\F_a(X)\to C^{**}$ defined by
$$
T_f(\gamma)(\phi)=\gamma(\phi\circ f)\qquad \left(\gamma\in\F_a(X) , \; \phi\in C^*\right)
$$
belongs to the set of bounded linear mappings from $\F_a(X)$ into $C^{**}$, and
\begin{align*}
\left\|T_f\right|=\sup_{\left\|\gamma\right|^{*}\leq 1}\left\|T_f(\gamma)\right|^{**}
                  &=\sup_{\left\|\gamma\right|^{*}\leq 1}\sup_{\left\|\phi\right|^{*}\leq 1}T_f(\gamma)(\phi)\\
                  &=\sup_{\left\|\phi\right|^{*}\leq 1}\sup_{\left\|\gamma\right|^{*}\leq 1}\gamma(\phi\circ f)
                  =\sup_{\left\|\phi\right|^{*}\leq 1}\|\phi\circ f|_S
                  \leq \|f|_S.
\end{align*}
Observe that the last inequality is accomplished by taking into account that $\phi$ is linear and
\begin{align*}
\sup_{\left\|\phi\right|^{*}\leq 1}\|\phi\circ f|_S
&=\sup_{\left\|\phi\right|^{*}\leq 1}\sup_{d(y,x)>0}\left\{\frac{(\phi\circ f)(x)-(\phi\circ f)(y)}{d(y,x)}\right\}\\
&=\sup_{\left\|\phi\right|^{*}\leq 1}\sup_{d(y,x)>0}\left\{\frac{\phi(f(x)-f(y))}{d(y,x)}\right\}
\leq \sup_{\left\|\phi\right|^{*}\leq 1}\left\|\phi\right|^{*}\|f|_S
=\|f|_S.
\end{align*}
(By abuse of notation, we still denote by $\left\|T_f\right|=\sup_{\left\|\gamma\right|^{*}\leq 1}\left\|T_f(\gamma)\right|^{*}$ the conic-norm of the linear function $T_f\colon\F_a(X)\to C^{**}$).
Furthermore, if $i_{C}\colon C\to C^{**}$ is the canonical injection, we have
\begin{align*}
\langle T_f (\delta(x)),\phi\rangle=T_f(\delta(x))(\phi)&=\delta(x)(\phi\circ f)\\
&=\phi(f(x))=i_{C}(f(x))(\phi)=\langle i_{C}(f(x)), \phi \rangle
\end{align*}
for every $x\in X$ and $\phi\in C^*$, and hence $T_f(\delta(x))=i_{C}(f(x))\in i_{C}(C)$ for every $x\in X$. 
This yields that $T_f(\gamma)\in i_{C}(C)$ for every $\gamma\in\F_a(X)$. Identifying $i_{C}(f(x))\in i_{C}(C)$ with $f(x)\in C$, we have $T_f\in \mathcal{L}(\F_a(X), C)$ and $T_f\circ \delta=f$. So, since $T_f\circ \delta=f$ and $\delta$ is an isometry (Proposition~\ref{isoinj}), we deduce that
\begin{align*}
\|f|_S&=\sup_{d(y,x)>0}\left\{\frac{d_e^c(f(y),f(x))}{d(y,x)}\right\}\\
&=\sup_{d(y,x)>0}\left\{\frac{\|T_f (\delta(x))-T_f(\delta(y))|}{d(y,x)}\right\}=\sup_{d(y,x)>0}\left\{\frac{\|T_f(\delta(x)-\delta(y))|}{d(y,x)}\right\}\\
&\leq\sup_{d(y,x)>0}\left\{\frac{\|T_f|\|\delta(x)-\delta(y)|^{*}}{d(y,x)}\right\}
=\|T_f|\sup_{d(y,x)>0}\left\{\frac{\|\delta(x)-\delta(y)|^{*}}{\|\delta(x)-\delta(y)|^{*}}\right\}
=\|T_f|.
\end{align*}
Thus $\left\|T_f\right|=\|f|_S$.
Assume now that there exists a linear bounded mapping $S_f:\F_a(X)\to C$ such that $S_f\circ\delta=f$. Then it is clear that $S_f(\delta(x))=T_f(\delta(x))$ for all $x\in X$ and, by the definition of $\F_a(X)$, it follows that $S_f=T_f$.
\end{proof}
\begin{remark}[Universal property]
Equivalently, the condition $T_f \circ \delta=f$ means that the following diagram commutes
$$
\begin{tikzpicture}
  \node (X) {$X$};
  \node (FX1) [below of=X] {};
 \node (FX) [below of=FX1] {$\F_a(X)$};
  \node (E1) [right of=FX] { };
  \node (E) [right of=E1] {$C$};
  \draw[->] (X) to node {$f$} (E);
  \draw[->] (X) to node [swap] {$\delta$} (FX);
  \draw[->, dashed] (FX) to node [swap] {$T_f$} (E);
\end{tikzpicture}
$$
Furthermore, as a consequence of the universal property that we have just proved, it is not difficult to establish that the mapping $f\mapsto T_f$ is an isometric isomorphism of $\SLip_0(X,C)$ into the cone of bounded linear mappings $L(\F_a(X),C)$, which constitutes another proof of Theorem \ref{predual} for the particular case $C=\mathbb{R}$. Indeed, we already know that the mapping $f\mapsto T_f$ is an isometry of $\SLip_0(X,\mathbb{R})$ onto $\F_a(X)^*$. Now, given $T\in L(\F_a(X),C)$, we can define a mapping $f\colon X\to C$ by $f(x)=T(\delta(x))$ for all $x\in X$. Since
$$
d_e^c(f(y),f(x))=d_e^c(T(\delta(y)),T(\delta(x)))
\leq \left\|T\right|\|\delta(x)-\delta(y)|^{*}
=\left\|T\right|d(y,x)
$$
for all $x,y\in X$, the function $f$ is in $\SLip_0(X,C)$. By the universal property of $\F_a(X)$, there is a unique operator $T_f\in L(\F_a(X),C)$ such that $T_f\circ \delta=f$. Hence $T=T_f$ and thus the mapping $f\mapsto T_f$ is a surjective isometry.
\end{remark}

The proof of the following result is immediate from Theorem \ref{Lslf}.
\begin{corollary}[Linearization of quasi-metric morphisms]\label{linecor}
Let $(X_1,d_1)$ and $(X_2,d_2)$ be two pointed quasi-metric spaces, and $f\in \mathrm{SLip}_0(X_1,X_2)$. Then there is a unique linear map $\hat{T}_f:\F_a(X_1)\to \F_a(X_2)$ such that $\hat{T}_f\circ \delta_{X_1}=\delta_{X_2}\circ f$, i.e. the diagram
$$
\begin{tikzpicture}
  \node (X) {$X_1$};
  \node (FX1) [below of=X] {};
 \node (Y) [right of=X] {};
 \node (Z) [right of=Y] {$X_2$ };
  \node (FX) [below of=FX1] {$\F_a(X_1)$};
  \node (E1) [right of=FX] { };
  \node (E) [right of=E1] {$\F_a(X_2)$};
  \draw[->] (X) to node [swap] {$\delta_{X_1}$} (FX);
  \draw[->, dashed] (FX) to node [swap] {$\hat{T}_f$} (E);
  \draw[->] (X) to node {$f$} (Z);
  \draw[->] (Z) to node [swap] {$\delta_{X_2}$} (E);
\end{tikzpicture}
$$
commutes,
and $\|\hat{T}_f|= \|f|_S$, where $\delta_{X_1}$ and $\delta_{X_2}$ are the isometric injections of the quasi-metric spaces $(X_1,d_1)$ and $(X_2,d_2)$ to their free spaces (\emph{c.f.} Proposition \ref{isoinj}).
\end{corollary}
For the following proposition, we refer to the reader to \cite{Ext} for a survey on the extensions of semi-Lipschitz functions on quasi-metric spaces.
\begin{proposition}[The free space of a quasi-metric subspace] \label{prop-mc}
Let $(X,d)$ be a quasi-metric space with base point $x_0$, and consider $(M,d)$ a subspace of $(X,d)$ such that $x_0\in M$. Then $\F_a(M)$ is
isometrically isomorphic to a subspace of $\F_a(X)$.
\end{proposition}
\begin{proof}
Let $\hat{T}_i:\F_a(M)\to \F_a(X)$ be the linearization given by Corollary \ref{linecor} of the identity mapping $i:M\to X$. Since $\|\hat{T}_i|=\|i|_S=1$, we know that $\|\hat{T}_i(Q)|_{\F_a(X)}^\ast\leq \|Q|_{\F_a(M)}^\ast$. For the opposite inequality, consider $Q\in \mathrm{span}\left(\delta(M)\right)$. Clearly, $\hat{T}_i(Q)=Q\in \mathrm{span}\left(\delta(X)\right)$. Then, for any $f\in \SLip(M)$, the expression  $\tilde{f}(x)=\inf_{m\in M}\{f(m)+\|f|_Sd(m,x)\}$, $x\in X$ (which is an adaptation of the McShane extension of Lipschitz maps), provides a semi-Lipschitz extension with the same associated conic-norm $\|f|_S$. It follows that
$$\|Q|_{\F_a(M)}^\ast=\sup_{\substack{\|f|_S\leq 1\\ f\in \SLip(M)} }\langle Q,f\rangle\leq\sup_{\substack{\|f|_S\leq 1\\ f\in \SLip(X)}} \langle Q,f\rangle =\|Q|_{\F_a(X)}^\ast=\|\hat{T}_i(Q)|_{\F_a(X)}^\ast.$$
By continuity of $\hat{T}_i$ (and density of $\mathrm{span}\left(\delta(M)\right)$ in $\F_a(M)$), we can extend the previous inequality to any $Q\in \F_a(M)$, which concludes the proof.
\end{proof}

Let us consider another conic-norm on $\mathrm{span}\left(\delta(X)\right)$ (and on $\F_a(X)$) which is based on a variant of the so-called Kantorovich-Rubinstein norm (see \cite[Section 8.4.5]{cobzas2019lips}).

\begin{example}[Kantorovich-Rubinstein conic-norm]
Let $X\neq\emptyset$ be a set equipped with a quasi-metric $d$ and a base point $x_0$. For $\gamma,\overline{\gamma}\in \mathrm{span}\left(\delta(X)\right)$ consider all representations of the form \linebreak
 $\gamma-\overline{\gamma}=\sum_{i=1}^n\lambda_i(\widehat{y}_i-\widehat{z}_i)$, where $\lambda_i\geq 0$ and possibly some $\widehat{y}_i$ or $\widehat{z}_i$ are equal to $\widehat{x}_0=0$, and set
$$d_{KR}(\gamma,\overline{\gamma}):=\inf\{\lambda_1d(z_1,y_1)+\ldots+\lambda_nd(z_n,y_n)\}.$$
Then $\|\gamma|_{KR}:=d_{KR}(\widehat{x}_0,\gamma)$ is an asymmetric norm on $\mathrm{span}\left(\delta(X)\right)$ and
$$d_{KR}(\widehat{x},\widehat{y})=d(y,x),\ \mbox{for all }x,y\in X.$$
Moreover, $\|\gamma|_{KR}$ coincides with the restriction of the conic norm $\|\cdot|^*$ of $\mathrm{SLip}_0(X)^*$ to $\mathrm{span}\left(\delta(X)\right)$ and thus extends to  $\F_a(X)$.
Indeed, if $\|\cdot|'$ is a conic-norm on $\mathrm{span}\left(\delta(X)\right)$ satisfying $\|\delta(x)-\delta(y)|'\leq d(y,x)$, for all $x,y\in X$, then every $\gamma=\lambda_1(\widehat{y_1}-\widehat{z_1})+\ldots+\lambda_n(\widehat{y_n}-\widehat{z_n})$ accomplishes
\begin{align*}
\|\gamma|'=\|\lambda_1(\widehat{y_1}-\widehat{z_1})+\ldots+\lambda_n(\widehat{y_n}-\widehat{z_n})|'&\leq
\|\lambda_1(\widehat{y_1}-\widehat{z_1})|'+\ldots+\|\lambda_n(\widehat{y_n}-\widehat{z_n})|'\\
 &\leq \lambda_1 d(z_1,y_1)+\ldots+\lambda_n d(z_n,y_n),
 \end{align*}
which shows that $\|\gamma|'\leq \|\gamma|_{KR}$. Particularly, we deduce from this that $\|\gamma|^{*}\leq \|\gamma|_{KR}$ (since the conic-norm $\|\cdot|^{*}$ on $\F_a(X)$ satisfies $\|\delta(x)-\delta(y)|^{*}= d(y,x)$, for all $x,y\in X$). Hence
$d(y,x)=\|\delta(x)-\delta(y)|^{*}\leq \|\delta(x)-\delta(y)|_{KR}\leq d(y,x),$ for all $x,y\in X$, which implies that
$$\|\delta(x)-\delta(y)|_{KR}= d(y,x),\ \ \mbox{for all } x,y\in X.$$
Consider now the mapping $L:X\to \left(\mathrm{span}\left(\delta(X)\right),\|\cdot|_{KR}\right)$ sending $x$ to $\delta(x)$, which is clearly an isometric embedding.
By the universality property of $\F_a(X)$ (see Theorem \ref{Lslf}), we know that $L$ extends to $\tilde{L}:\F_a(X)\to\left(\mathrm{span}\left(\delta(X)\right),\|\cdot|_{KR}\right)$ and $\|\cdot|_{KR}\leq \|\cdot|^{*}$, so the conic-norms $\|\cdot|_{KR}$ and $\|\cdot|^{*}$ are the same.
\end{example}

\section{Examples of semi-Lipschitz free spaces}\label{examples}
\smallskip
\noindent Let us now illustrate the semi-Lipschitz free space for three concrete examples of quasi-metric spaces: a finite quasi-metric space consisting of three points, the set of natural numbers $\mathbb{N}$ with a discrete quasi-metric and the set of real numbers
$\mathbb{R}$ under the quasi-hemi-metric defined by the canonical conic hemi-norm $u$. We also include an example-scheme stemming from canonical asymmetrizations of subsets of $\R$-trees.

\subsection{A 3-point quasi-metric space} Let $X=\{x_0,x_1,x_2\}$ be a set of three points, endowed with the following quasi-metric (in a general form):
    \begin{align*}
         &\rho(x_0,x_1)=a_{01} &\rho(x_1,x_0)=a_{10}& &\rho(x_0,x_2)=a_{02}\\
         &\rho(x_2,x_0)=a_{20} &\rho(x_1,x_2)=a_{12}& &\rho(x_2,x_1)=a_{21}
    \end{align*}
Taking $x_0$ as base point, it is clear that the set of semi-Lipschitz functions vanishing at $x_0$ can be algebraically identified with $\mathbb{R}^2$, i.e. any function $g:X\to\mathbb{R}$ with $g(x_0)=0$ is in $\SLip_0(X)$, with associated semi-Lipschitz norm equal to
$$\|g|_S=\max\Big\{\,\frac{g_1-g_2}{a_{21}},\,\frac{g_2-g_1}{a_{12}},\,\,\frac{g_1}{a_{01}},\,\,\frac{g_2}{a_{02}},\,\,\frac{-g_1}{a_{10}},\,\,\frac{-g_2}{a_{20}}\,\Big\},$$
where $g_1=g(x_1)$ and $g_2=g(x_2)$.
Therefore, the unit ball $B$ of $\mathrm{SLip}_0(X,\rho)\simeq\mathbb{R}^2$ is the polygon generated by the linear inequalities defined in terms of the asymmetric norm. The dual cone of $(\SLip_0(X),\|\cdot|_S)$ is the vector space $\mathbb{R}^2$ endowed with the asymmetric norm determined by the Minkowski gauge of the asymmetric polar $B^o$ of the unit ball $B$ of $\mathrm{SLip}_0(X,\rho)$, that is
$$B^o=\{X\in \mathbb{R}^2:\: \langle g,X\rangle \leq 1,\,\,\forall g\in B\}.$$
\smallskip
Since the evaluation functionals $\delta(x_1), \delta(x_2)$ generate the vector space $\mathbb{R}^2$, it follows that $\mathcal{F}_a(X,\rho)$ is isomorphic to $\mathbb{R}^2$, with the asymmetric norm determined by the aforementioned Minkowski gauge.
Furthermore, for any $g\in\SLip_0(X)$, its linearization  $T_g\colon\F_a(X)\to \mathbb{R}$ is given by
$$
T_g\left(\lambda_1\hat{x_1}+\lambda_2\hat{x_2}\right)=\lambda_1g(x_1)+\lambda_2g(x_2),
$$
with $\lambda_i\in\mathbb{R}$, $i=1,2$. Notice that the unit balls of $\mathrm{SLip}_0(X,\rho)$ and its dual cone have at most $6$ extreme points (see Figure \ref{fig2}).

\begin{figure}[h]
\centering
\includegraphics[width=6.3cm]{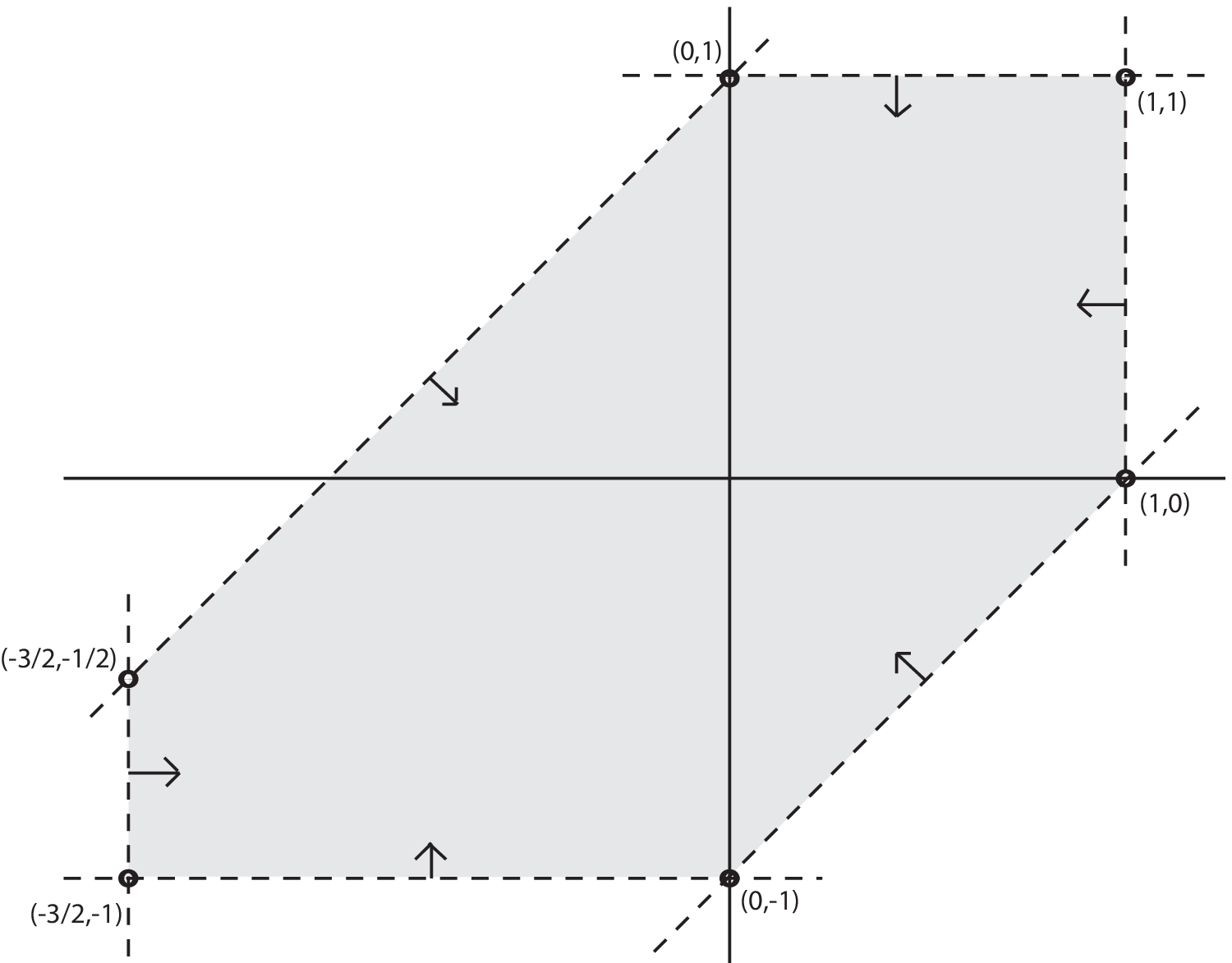}\hspace{1cm}
\includegraphics[width=5.13cm]{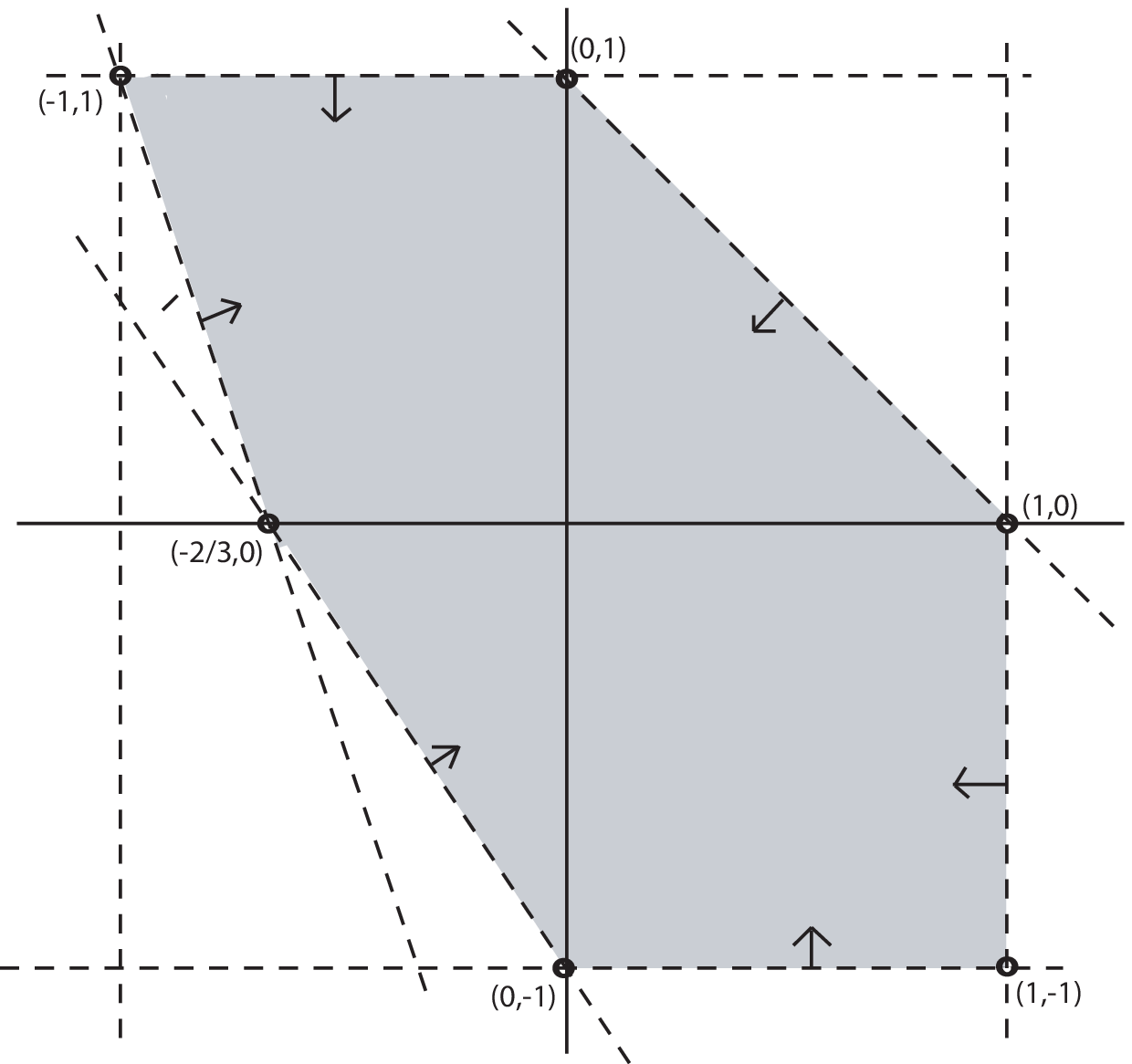}
\caption{Representation of the unit ball of $\mathrm{SLip}_0(X,\rho)$ and its asymmetric polar, respectively, with $X=\{x_0,x_1,x_2\}$, $\rho(x_1,x_0)=\frac{ 3}{ 2}$ and $\rho(x_i,x_j)=1$ for $i\neq j$ with $(i,j)\neq (1,0)$}\label{fig2}
\end{figure}


\subsection{\texorpdfstring{$\mathbb{N}$}{N} as a metric or quasi-metric space}
We now consider the set of natural numbers $\mathbb{N}$ (including~$0$) endowed with the quasi-metric defined by

\begin{equation*}
d(n,m)=\begin{cases} 1, &\hbox{if}\ m\notin\{0,n\} \\
0,\ &\hbox{if } m\in\{0,n\}
\end{cases}.
\end{equation*}
\noindent We fix as a base point $x_0=0$. Let $y=\left(y(n)\right)_n\in \mathrm{SLip}_0(\mathbb{N},d)$. Then $y(0)=0$ and the semi-Lipschitz condition implies that the sequence $\left(y(n)\right)_n$ is non-negative: indeed
$$y(0)-y(n)=-y(n)\leq \|y|_S\,d(n,0)=0$$
and
$$y(n)-y(0)=y(n)\leq \|y|_S \,d(0,n)=\|y|_S.$$ Therefore we have $\left(y(n)\right)_n\in \ell^\infty(\mathbb{N})$ and $\|y|_S\geq \|y\|_\infty$. Moreover,
$$\|y|_S=\sup_{d(n,m)>0}\frac{y(m)-y(n)}{1}\leq \sup_{d(n,m)>0}y(m)=\|y\|_\infty,$$
\noindent since $y(n)\geq 0$ for all $n\in \mathbb{N}$. It is easy to check than any bounded non-negative sequence satisfies the semi-Lipschitz condition, so it follows that $\mathrm{SLip}_0(\mathbb{N},d)$ is  $(\ell^\infty_+(\mathbb{N}),\|\cdot\|_\infty)$, the positive cone of $\ell^\infty(\mathbb{N})$. The dual norm on $\ell^\infty_+(\mathbb{N})^*$ is given by
$$\|\varphi|^*=\sup_{\substack{ (y_n)\in \ell^\infty_+(\mathbb{N})\\ \|(y_n)\|_\infty\leq1\\}}\varphi((y_n)) .$$
The set of evaluation functionals $\{\delta(n):\:n\in \mathbb{N}\}\subset \ell^\infty_+(\mathbb{N})^*$ can be identified with the canonical basis of $\ell^1(\mathbb{N})$, so the linear span of $\delta(\mathbb{N})$ is the set of finitely supported sequences $c_{00}(\mathbb{N})$. On this set, the dual norm of $\mathrm{SLip}_0(\mathbb{N},d)^*$ becomes
$$\|(x_n)|^*=\sum_{n\in \mathbb{N}} \max\{x_n, 0\}\,=\sum_{n\in \mathbb{N}} x_n^+\,:=\,\|(x_n)|_{1,+},$$
\noindent since the supremum on the dual norm is taken over the positive cone of $\ell^\infty(\mathbb{N})$ (and it is attained at the sequence $\left(\mathrm{sgn}(x_n)\vee 0\right)$). It is easy to check that the symmetrization of the asymmetric norm $\|\cdot|_{1,+}$ is equivalent to the usual norm of $\ell^1(\mathbb{N})$, and therefore the asymmetric normed space $(\ell^1(\mathbb{N}),{\|\cdot|_{1,+}})$ satisfies the conditions to be the bicompletion of $(c_{00}(\mathbb{N}),{\|\cdot|_{1,+}})$. Therefore, the semi-Lipschitz free space $\mathcal{F}_a(\mathbb{N},d)$ is isometrically isomorphic to $(\ell^1(\mathbb{N}),\|\cdot|_{1,+})$ and the linearization $T_y$ of a function $y=\left(y(n)\right)_n\in \mathrm{SLip}_0(\mathbb{N},d)$ can be obtained from
$$
T_y\left(e_n\right)=y(n),\ n=1,2,\ldots
$$
where $e_n$ is the $n$-th element of the canonical basis of $\ell^1(\mathbb{N})$. 

\medskip

\noindent It is well known that the free space $\mathcal{F}(\mathbb{N},D)$ of $\mathbb{N}$ equipped with the distance
$$D(m,n)=\begin{cases}
2, &\hbox{if}\ \ n\notin\{0,m\}\\
1, &\hbox{if}\ \ n=0\mbox{ or }m=0\\
0, &\hbox{if}\ \ n=m
\end{cases}$$
is isometric to $\ell^1(\mathbb{N})$ (see, for instance, \cite{Godard,G_2015,Weaver}), and $$L=\Lip_0(\mathbb{N},D)=\{y=y_n\in\mathbb{R}^{\mathbb{N}}:\|y\|_L:=\frac{y(n)-y(m)}{D(m,n)}<\infty\}$$ is isometric to $\ell^{\infty}(\mathbb{N})$. Given $m,n\in\mathbb{N}$, then the canonical asymmetrization of $D$ (Definition \ref{canoasym}) is
$$D_+(m,n)=\|\widehat{\delta}(n)-\widehat{\delta}(m)|_{\mathcal{F}_+}=\sup_{\substack{y\in l^{\infty}(\N)_+\\\|y\|_{\infty}\leq 1}}\langle y,e_n-e_m\rangle=\sup_{0\leq y_n\leq 1}\sum_{k\geq 0}y_kx_k,$$
where $x_n\!=\!1$, $x_m=\!-\!1$, and $x_k=0$, for $k\notin\{n,m\}$. According to Theorem~\ref{Nino}, notice that \linebreak $\mathcal{F}(\mathbb{N},D)=\mathcal{F}(\mathbb{N},d)$ (as a set), with $d=D_+$, $\SLip_0(\mathbb{N},d)=\ell_{+}^{\infty}(\mathbb{N})=\SLip_0(\mathbb{N},D_+)$ and $\|x|_{\mathcal{F}_a}=\|x|_{\mathcal{F}_+}=\sum_{n\geq 0}x_n^+$.

\medskip

\subsection{The quasi-metric space \texorpdfstring{($\mathbb{R},u)$}{(R,u)}}\label{ss-5.3}
Note that since the symmetrized distance $d_u^{s}$ is equal to the usual metric of $\R$ (which can be seen as a pointed $\R$-tree),
$\F(\R,u)$ can be obtained from Proposition~\ref{SS*}. We hereby include a direct self-contained proof, which does not rely on Godard's work on $\R$-trees. Let us start with some preliminary results.

\begin{lemma}[Semi-Lipschitz functions in $( \R,u)$]
Let $f\in \mathrm{SLip}_0(\mathbb{R},u)$. Then $f$ is a non-decreasing function in $\Lip_0(\mathbb{R})$.
\end{lemma}
\begin{proof}
By Remark \ref{lip}, $f$ is Lipschitz on $(\mathbb{R},u^{s})=(\mathbb{R},|\cdot|)$, and therefore is differentiable almost everywhere. Note that if $x\leq y$, then $d_u(y,x)=0$, so $f(x)\leq f(y)$. As $f$ is non-decreasing, $f'\geq 0$.
\end{proof}

We are now ready to establish the following result.
\begin{lemma}\label{slip0}
The normed cone $(\mathrm{SLip}_0(\mathbb{R},u),\|\cdot|_S)$ is isometrically isomorphic to $(\mathcal{L}^\infty_+(\mathbb{R}), \|\cdot\|_\infty )$.
\end{lemma}
\begin{proof}
Consider the mapping $T:(\mathcal{L}^\infty_+(\mathbb{R}), \|\cdot\|_\infty) \to(\mathrm{SLip}_0(\mathbb{R},u),\|\cdot|_S)$ defined by $$Tg(x)=\int_0^x gd\lambda=\int \mathbf{1}_{[0,x]}g,$$
\noindent which is surjective by the previous analysis.
This mapping is well defined since  for $x\geq y$ we have $Tg(x)-Tg(y)=\int_y^x gd\lambda \leq \|g\|_\infty(x-y)=\|g\|_\infty d_u(y,x)$. If $x<y$ then $$Tg(x)-Tg(y)=-\int_x^y gd\lambda\leq 0=d_u(y,x).$$ This also proves that $\|Tg|_S\leq \|g\|_\infty$. On the other hand, consider $x\in \mathbb{R}$ a point of differentiability of $Tg$. Then $$Tg'(x)=\lim_{y\searrow x}\frac{Tg(y)-Tg(x)}{y-x}\leq \sup_{x< y}\frac{Tg(y)-Tg(x)}{y-x}=\|Tg|_S,$$
and since clearly $(Tg)'=g$, we conclude that $\|g\|_\infty\leq\|Tg|_S $ and that $T$ is an isometric isomorphism.
\end{proof}
\medskip

For the following result, if $f\in \mathcal{L}^1(\mathbb{R})$ recall the notation $\|f|_{1,+}=\int_{\R} f^+d\lambda$, where $f^+(x)=\max\{f(x), 0\}$ and $\lambda$ denotes the Lebesgue measure, which was used in Lemma \ref{duall1}.

\begin{theorem}\label{SlipR}
The semi-Lipschitz free space $\mathcal{F}_a\left((\mathbb{R},u)\right)$ of the asymmetric hemi-normed space $(\mathbb{R},u)$ is isometrically isomorphic to $(\mathcal{L}^1(\mathbb{R}),\|\cdot|_{1,+})$.
\end{theorem}
\begin{proof}
By Lemma \ref{duall1}, we know that $(\mathcal{L}^1(\mathbb{R}),\|\cdot|_{1,+})$ is the asymmetric predual of $(\mathcal{L}^\infty_+(\mathbb{R}),\|\cdot\|)$. Therefore we only need to check that the isometry $T:(\mathcal{L}^\infty_+(\mathbb{R}), \|\cdot\|_\infty) \to(\mathrm{SLip}_0(\mathbb{R},u),\|\cdot|_S)$ defined in the previous proof is $(w^*$-$w^*)$-continuous, in which case Lemma \ref{adjunto} will give us an isometry between the preduals $\mathcal{F}_a(\mathbb{R},u)$ and $(\mathcal{L}^1(\mathbb{R}),\|\cdot|_{1,+})$. So, let $(g_\alpha)$ be a net on $\mathcal{L}^\infty_+(\mathbb{R})$ converging to $g$ in the $w^*$ topology induced by the predual $(\mathcal{L}^1(\mathbb{R}),\|\cdot|_{1,+})$, and take $x\in \mathbb{R}$ and the corresponding $\widehat{x}\in \mathcal{F}_a(\mathbb{R},u)$. Then
\begin{equation}\label{1}
    \langle Tg_\alpha ,\hat x\rangle =\int_0^x g_\alpha =\langle g_\alpha,\mathbf{1}_{[0,x]}\rangle\longrightarrow \langle  g, \mathbf{1}_{[0,x]}\rangle,
\end{equation}
\noindent by the $w^*$ convergence of $(g_\alpha)$. Now, for an arbitrary $\mu\in \mathcal{F}_a(X)$ we can take a sequence $(\mu_n)\subset \mathrm{span}(\delta(\mathbb{R}))$ such that $\mu_n\to \mu$ in the symmetrized topology of $\mathrm{SLip}_0(\mathbb{R},u)^*$, and therefore
\begin{equation}\label{2}
 \langle Tg_\alpha, \mu\rangle = \lim_n\langle Tg_\alpha, \mu_n\rangle,
\end{equation}
where the last convergence is with respect to the usual norm on $\mathbb{R}$, thanks to the symmetrized-$|\cdot|$ continuity of semi-Lipschitz functions. Equations \eqref{1} and \eqref{2} yield that $\langle Tg_\alpha, \mu \rangle \to \langle Tg,\mu \rangle$ for the norm topology in $\mathbb{R}$, so $T$ is $(w^*$-$w^*)$-continuous, and by Lemma \ref{adjunto} there exists an isometric isomorphism between $(\mathcal{F}_a(\mathbb{R},u), \|\cdot|^*)$ and  $(\mathcal{L}^1(\mathbb{R}),\|\cdot|_{1,+})$.
\end{proof}

As we show in Example~\ref{sorg}, $d_u(x,y)=u(y-x)$ is a canonical asymmetrization of $D(x,y)=|y-x|$ for the cone $P=\{\phi\in L:\phi'\geq 0\}$. Notice that the canonical asymmetrization $D_+$, based on the cone $P=L_+$ gives a different asymmetrization.

\subsection{Canonic asymmetrization of subsets of \texorpdfstring{$\R$}{R}-trees.}
Propositions~\ref{SS*} and~\ref{L1(T)} provide a variety of examples of quasi-metric spaces $(X,d)$ whose corresponding semi-Lipschitz free spaces are isometrically isomorphic to subspaces of $(\mathcal{L}^1(T),\Vert \cdot|_{1,+})$, where $T$ is an $\R$-tree containing the symmetrized space~$(X,d^{s})$. We can obtain more specific examples by applying the following recent result from \cite[Theorem~1.1]{APP}, which gives a characterization of all complete metric spaces whose Lipschitz-free space is isometric to a subspace of $\ell^1(\Gamma)$ for some set~$\Gamma$.

\begin{theorem}\label{APPtrees}
Let $(X,D)$ be a complete pointed metric space. Then the following are equivalent:
\begin{itemize}
\item[(i).] $\F(X)$ is isometrically isomorphic to a subspace of
$\ell^1(\Gamma)$ for some set $\Gamma$;
\item[(ii).] $(X,D)$ is a subset of an $\R$-tree such that $\lambda(X)=0$ and $\lambda(\overline{\mathrm{Br}(X)})=0$, where $\lambda$ is the length measure and $\mathrm{Br}(X)$ is the set of branching points of $X$.
\end{itemize}
\end{theorem}
\noindent Since every metric space as above satisfies property $\mathbf{(S^\ast)}$ ({\em c.f.} Proposition~\ref{L1(T)}), we deduce that the corresponding semi-Lipschitz free space are isometrically isomorphic to $(\ell^1(\Gamma), \Vert\!\cdot\!|_{1,+})$ for some set
$\Gamma$.
\medskip

A careful reader might have observed that in all examples presented in this section, the semi-Lipschitz free space of the given quasi-metric space can be easily obtained from the Lipschitz free space of its symmetrization. We shall now show that this is always the case, provided assumption ($\mathcal{H}$) below holds. (This is the case in all of the aforementioned examples).  \smallskip\newline
Using the same notation as in the second part of Subsection~\ref{sub3.3}, let $(X,d)$ be a quasi-metric space and $(X,D)$ its symmetrization ($D$ is either $d^s$ or $d^{s_0}$). Then $P:=\SLip_0(X,d)$ is a cone in $\Lip_0(X,D)$ and $\|\phi\|_L\leq \|\phi|_{S}$ for all $\phi\in P$. Let us assume:
\begin{itemize}
\item[($\mathcal{H}$)] For every $\phi \in \Lip_0(X,D)$ there exist $\phi _{1},\phi _{2}\in \SLip_0(X,d)$ such that
$$\phi =\phi _{1}-\phi _{2}\qquad \text{ and}\qquad   \max\left\{ \|\phi _{1}|_{S},\|\phi _{2}|_{S}\right\} \,\leq \,\|\phi\|_L.$$
\end{itemize}
Notice that since $\|\phi_i\|_L\leq\|\phi_i\|_S$, for $i\in\{1,2\}$ and in view of the triangular inequality
$$\|\phi\|_L=\|\phi_1-\phi_2\| \leq \,||\phi _{1}||_{L}+||\phi _{2}||_{L}$$ we observe that ($\mathcal{H}$) yields in particular that $P$ induces a canonical asymmetrization in $\F(X,D)$ (in the sense of Remark~\ref{naturalasymmetrizednorm}).
\begin{proposition}
Let $(X,d)$ be a quasi-metric space and assume ($\mathcal{H}$) holds. Then the semi-Lipschitz free space $\F_a(X,d)$ coincides (as a set) with the free space $\F(X,D)$ of the symmetrized space $(X,D)$ and is endowed with the asymmetric norm
$$\|Q|\:=\sup_{\substack{\|\phi|_S\leq 1\\\phi \in \SLip_0(X,d)}}\langle Q,\phi\rangle\,,\quad \text{for all } \,Q\in \F_a(X,d).$$
\end{proposition}
\begin{proof}
Following the method used in Subsection~\ref{sub3.3}, we start by identifying the sets
\begin{equation*}
F=\mathrm{span}\{\delta (x):x\in X\}\subset \SLip_0(X,d)^{\ast }
\end{equation*}
and
\begin{equation*}
\widehat{F}=\mathrm{span}\{\widehat{\delta}(x):x\in X\}\subset \Lip_0(X,D)^{\ast }
\end{equation*}
where $\delta $ and $\widehat{\delta}$ are the canonical injections of
$(X,d)$ into $\SLip_0(X,d)^{\ast }$ and of $(X,D)$ into $\Lip_0(X,D)^\ast $, respectively. To conclude, it suffices to prove that the $d^s$-symmetrization $\|\cdot\|^s$ of the asymmetric norm $\|\cdot|$ is equivalent to $\|\cdot\|_{\F}$. Consider $Q\in F$. Since $\|\phi|_S\geq \|\phi\|_L$ for any $\phi \in \SLip_0(X,d)$, it follows (by the definition of each norm) that $\|Q| \leq \|Q\|_{\F}$, so $\|Q\|^{s} \leq 2\|Q\|_{\F}$. Conversely, take $\phi$ in the unit ball of $\Lip_0(X,D)$ such that $\|Q\|_{\F}=\langle Q, \phi\rangle$, and consider $ \phi_{1},\phi_{2}\in \SLip_0(X,d)$ such that $\phi =\phi_{1}-\phi_{2}$, with $\max\left\{ \|\phi _{1}|_{S},\|\phi _{2}|_{S}\right\}\leq\|\phi\|_L\leq 1$. Then
$$\|Q\|_{\F}=\langle Q, \phi\rangle=\langle Q, \phi_1\rangle + \langle -Q, \phi_2\rangle \leq \|Q| + \|-Q|:=\|Q\|^s.$$
The result follows from the fact that $\mathcal{F}_{a}(X,d)=\overline{F}^{\| \cdot \|^{s}_{\mathcal{F}
_{a}}}=\overline{F}^{\Vert \cdot \Vert _{\mathcal{F}}}=\mathcal{F}(X,D).$
\end{proof}

\section{Conclusions, future research}\label{sec:6}

 It seems to be of paramount importance to relate symmetric and asymmetric structures in a way that remains compatible with the embeddings to the corresponding free spaces. At the same time, a given asymmetric space might not be equal to a canonical asymmetrization of some (symmetric) metric space. It is even unknown if a given quasi-distance is always topologically equivalent to
a quasi-distance with this property.
Let us observe that there are several ways to symmetrize a quasi-distance and obtain (symmetric) metric spaces that generate the same underlying topology. Indeed, let $\phi:\mathbb{R}_+^2\to [0,\infty)$ be any nonnegative continuous function satisfying
$\phi(a,b)\geq \max\{a,b\}$ and assume further that $\phi$ is symmetric ({\em i.e.} $\phi(a,b)=\phi(b,a)$, for all $a,b\geq 0$), coercive ({\em i.e.} bounded sublevel sets) and $\phi(a,b)=0$ if and only if $a=b$, for every $a,b\geq 0$. Then for every quasi-distance $d$, we obtain a symmetric distance $d^{\phi}$ via the formula:
$$d^{\phi}(x,y)=\phi(d(x,y),d(y,x)),\quad\text{for all }\, x,y\in X.$$ (In Definition~\ref{sym-dis} we have been focused only on the cases $\phi_0(a,b)=\max\{a,b\}$ and $\phi(a,b)=a+b$). Concerning the inverse procedure (asymmetrization), we have mainly been based on the lattice structure of the nonlinear dual of the metric spaces, which is used to induce a canonical asymmetrization on the initial metric space
$(X,D)$. However, in some cases,  canonical asymmetrizations of the same space may look completely different (see Example~\ref{example_S}).  In a similar spirit, starting from an asymmetric space $(X,d)$ and considering its symmetrization $(X,D)$, it is not known whether or not the set of semi-Lipschitz functions $\SLip_0(X,d)$, viewed as a cone in $\Lip_0(X,D)$, induces a canonical asymmetrization on $(X,D)$ ({\em c.f.} Proposition~\ref{compatibilityIII}) and in particular when conditions~\eqref{eq-cII} or
($\mathcal{H}$) hold. Therefore, many natural questions still remain unexplored and the whole panorama is far from being completely understood. \smallskip

Another topic that merits to be further explored is the particular case of normed spaces. Indeed, considering a normed space $(X,\|\!\cdot\!\|)$ as metric space, leads to a canonical asymmetrization $D_P$ of its distance $D(x,y)=\|y-x\|$. If $X=\R$ (therefore, $\mathcal{F}(\R)=\mathcal{L}^1(\R)$), then taking $P=\mathcal{L}_{+}^{\infty}(\R)$ we observe that the asymmetrized distance $D_P=u$ is associated to an asymmetric norm, see Subsection~\ref{ss-5.3}. It would be interesting to determine which normed spaces admit canonical asymmetrizations of their norms in this way, and inversely, characterize asymmetric norms that can be obtained via this procedure.\smallskip

Let us finish this discussion with a more philosophical comment. Convexity is a fundamental notion of Variational Analysis whose definition relies on a linear structure. It particular it is not distance--related, in strong contrast with usual differential calculus, Lipschitz functions, Riemann/Finsler geometry and metric generalizations of convexity. In particular, the class of Lipschitz functions is clearly affected if we consider asymmetric distances or asymmetrizations of the space. In both cases semi-Lipschitz functions are the adequate morphisms to describe properties of the space (see~\cite{DJV} {\em e.g.}) and this work outlines a natural way to define a notion of a quasi-metric free space as well as of a canonical asymmetrization of a space. Let us recall that in a metric space $(X,D)$, every Lipschitz function $\varphi:K\mapsto \R$ can be extended to a Lipschitz function $\tilde{\varphi}:X\mapsto \R$ without any increase in the Lipschitz constant, that is $\Lip(\varphi,K)=\Lip(\tilde{\varphi},X)$. Indeed, McShane, in \cite{McShane}, gave a concrete formula (based on inf-convolution) to obtain such extension, called minimal extension. In a completely analogous way one constructs minimal semi-Lipschitz extensions for real-valued semi-Lipschitz functions defined on a quasi-metric space $(X,d)$ maintaining the semi-Lipschitz constant  (see proof of Proposition~\ref{prop-mc}). An important instance of minimal Lipschitz extension is the so-called AMLE (absolutely minimal Lipschitz extension), which in case of a Euclidean space corresponds to the solution of the infinite-Laplacian operator, see \cite{Jensen1993}. The notion of absolutely minimal semi-Lipschitz extension makes perfect sense in an asymmetric framework, but it is not known whether or not such an extension always exists and if, in case of a finite dimensional asymmetric normed space, it can be identified to a solution of some differential-type operator. Concerning this latter topic, for the time being there is no clear way to deal with differentiability in asymmetric structures. Doing this in a satisfactory manner seems to be somehow related to the fact that the asymmetric space is canonical, since in this case, and only there, one would expect to obtain a canonical asymmetric differential calculus. Formalizing and proving this meta-theorem is mathematically challenging, but at the same time it will shed new light on asymmetric analysis, topic on which very little is known and might be relevant in the future. \smallskip

\subsection*{Acknowledgement} The authors wish to express their gratitude to the referee for his careful reading and comments, as well as for his suggestion to use Godard's characterization of isometric embeddings to an $\R$-tree, which led to Propositions~\ref{L1(T)} and~\ref{Rtrees}. \smallskip\newline
Part of this work was realized during a research stay of A. Daniilidis (``Gaspard Monge'' invited professor) and F. Venegas (Research trainee) at INRIA (\'Equipe Tropicale) and CMAP of \'Ecole Polytechnique (France), from September to December 2018. The authors wish to thank G.~Godefroy and S.~Gaubert for insight comments. Part of this work has been presented by the third author at the conference ``Function Theory on Infinite Dimensional Spaces XVI", held at the Complutense University of Madrid (November~2019). This author is grateful to the organizing committee of this event for hospitality and financial support. He also thanks S.~Tapia, A.~Proch\'azka and C.~Petitjean for useful discussions.

\medskip

\noindent Shortly before the conclusion of this work, it came to our knowledge that L.~Candido, P.~Kaufmann and J.~Rom\~{a}o have been working on this topic, and independently obtained results, which partially overlap with some of the ones presented here, see \cite{BCK}.


\newpage

\noindent Aris Daniilidis

\medskip

\noindent DIM--CMM, UMI CNRS 2807\newline Beauchef 851, FCFM, Universidad de
Chile \smallskip

\noindent E-mail: \texttt{arisd@dim.uchile.cl};\, \, \newline
\noindent \texttt{http://www.dim.uchile.cl/{\raise.17ex\hbox{$\scriptstyle\sim$}}arisd}

\medskip

\noindent Research supported by the grants: \newline CMM AFB170001,
FONDECYT 1211217,  ECOS-CONICYT C18E04 (Chile), \\
and PGC2018-097960-B-C22 (MCIU/AEI/ERDF, UE).

\vspace{0.8cm}

\noindent Juan Mat\'{i}as Sepulcre

\medskip

\noindent Departamento de Matem\'{a}ticas, Facultad de Ciencias, Ap. de Correos 99, 03080 Alicante
\newline Universidad de
Alicante \smallskip

\noindent E-mail: \texttt{JM.Sepulcre@ua.es};

\noindent \texttt{http://personal.ua.es/en/jm-sepulcre/}

\medskip

\noindent Research supported by the grant:

\noindent PGC2018-097960-B-C22 (MCIU/AEI/ERDF, UE).
\medskip

\vspace{0.8cm}

\noindent Francisco Venegas M.

\medskip

\noindent DIM--CMM, UMI CNRS 2807\newline Beauchef 851, FCFM, Universidad de
Chile \smallskip

\noindent E-mail: \texttt{venegas.francisco92@gmail.com};
\medskip

\noindent Research supported by the grants: \newline CMM AFB170001,
FONDECYT 1171854,  ECOS-CONICYT C18E04, \\
and CONICYT Doctorate Fellowship PFCHA 2019--21191167 (Chile).
\medskip

\bigskip
\begin{center}
\rule{16cm}{1.4pt}
\end{center}

\end{document}